\documentclass[11pt, a4paper]{amsart}

\usepackage{a4wide}
\usepackage{graphicx,enumitem}
\usepackage{amsmath,amssymb}
\usepackage[english]{babel}
\usepackage{subfigure}
\usepackage{units}
\usepackage{diagbox}
\usepackage{fontenc}
\usepackage{xcolor}
\usepackage{caption}
\usepackage{ulem}
\usepackage{cancel}
\usepackage{tikz}
\usepackage{amssymb}
\DeclareSymbolFontAlphabet{\amsmathbb}{AMSb}%
\usepackage{mathbbol}
\usepackage{latexsym}
\usepackage{xspace,bm}
\usetikzlibrary{positioning}
\usetikzlibrary{arrows,decorations.markings}
\usetikzlibrary {arrows.meta,bending}
\usetikzlibrary{external}
\usepackage[color=yellow]{todonotes}

\usepackage[pdftex,
bookmarksopen,
colorlinks,
linkcolor=black,
urlcolor=black,
citecolor=black
]{hyperref}
\hypersetup{pdfauthor={S. Ephrati, E. Jansson, A. Lang, E. Luesink}}

\usepackage{cleveref}

\usepackage{dsfont}

\usepackage{cases}

\usepackage{mathtools}

\newcommand{\R}{\amsmathbb{R}}

\newcommand{\C}{\amsmathbb{C}}

\newcommand{\cO}{\mathcal{O}}

\DeclareMathOperator{\E}{\amsmathbb{E}} 

\DeclareMathOperator{\trace}{Tr}

\DeclareMathOperator{\Diag}{Diag}
\newcommand{\dd}{\mathrm{d}}

\newtheorem{lemma}{Lemma}[section]

\newtheorem{theorem}[lemma]{Theorem}
\theoremstyle{remark}
\newtheorem{remark}[lemma]{Remark}

\theoremstyle{definition}

\usepackage{color}
\definecolor{darkgreen}{rgb}{0,.6,0}

\newcommand{\sclass}[1]
{
  \small	
  \textbf{Mathematics Subject Classification } #1
}

\title[Implicit midpoint method for stochastic Lie--Poisson systems]{An exponential map free implicit midpoint method for stochastic Lie--Poisson systems}

\author{Sagy Ephrati}
\author{Erik Jansson}
\author{Annika Lang}
\address[Sagy Ephrati, Erik Jansson, Annika Lang]{Department of Mathematical Sciences, Chalmers University of Technology and University of Gothenburg, 412 96 Gothenburg, Sweden}
\email{\{sagy, erikjans, annika.lang\}@chalmers.se}
\author{Erwin Luesink}
\address[Erwin Luesink]{Korteweg-De Vries Institute, University of Amsterdam, PO Box 94248, Science Park 107, 1090 GE Amsterdam, The Netherlands}
\email{e.luesink@uva.nl}

\date{}

\begin{document}

\maketitle

\begin{abstract}
An integrator for a class of stochastic Lie--Poisson systems driven by Stratonovich noise is developed. The integrator is suited for Lie--Poisson systems that also admit an isospectral formulation, which enables scalability to high-dimensional systems. Its derivation follows from discrete Lie--Poisson reduction of the symplectic midpoint scheme for stochastic Hamiltonian systems. We prove almost sure preservation of Casimir functions and coadjoint orbits under the numerical flow and provide strong and weak convergence rates of the proposed method. The scalability, structure-conservation, and convergence rates are illustrated numerically for the (generalized) rigid body, point vortex dynamics, and the two-dimensional Euler equations on the sphere.
\end{abstract}

\sclass{60H35, 60H10, 65C30, 65P10, 37M15}

\section{Introduction}
Recent years have seen a growing interest in the development of geometric numerical methods in science and engineering, and the concurrent use of stochastic forcing to quantify uncertainty in dynamical systems. Geometric methods arise naturally in conservative mechanics that feature a Hamiltonian or Lagrangian structure, and aim to preserve numerically certain physical laws that arise from this mathematical structure. This paradigm has encouraged the development of novel methods for, e.g., plasma physics to research nuclear fusion \cite{kraus2021variational, kraus2017gempic}, and geophysical fluid dynamics to study planetary flows \cite{franken2023zeitlin, wimmer2020energy, zeitlin2004self}. Structure-preserving integrators have shown to provide accurate statistical results for long-time simulations of complex systems, progressing theoretical research in fluid dynamics \cite{modin2020casimir, modin2022canonical} and numerically confirming statistical predictions in turbulent flows \cite{cifani2022casimir}. Stochastic fluid dynamics has been an active area of research over the last two decades \cite{flandoli2002probabilistic, albeverio2008introduction}. The wish to include stochasticity whilst retaining the mathematical structure of the underlying equations has motivated the study of transport noise. In \cite{flandoli2022additive}, it is shown that one can obtain noise of transport type also from additive noise. In the work by Holm \cite{holm2015variational} a certain type of transport noise referred to as `stochastic advection by Lie transport' (SALT) was introduced. This noise aims to model unknown or unresolvable dynamics as a stochastic forcing. Since its introduction, SALT has found meaningful applications in geophysical fluid dynamics with research consisting of model development \cite{holm2021stochastic, holm2023deterministic}, theoretical analysis \cite{flandoli2021scaling, flandoli2020convergence, goodair20223d}, and numerical studies on uncertainty quantification \cite{cotter2019numerically, ephrati2023data, cifani2022sparse} and data assimilation \cite{cotter2020particle, cotter2020data, lang2022bayesian}.

We develop a new numerical integrator for a class of Lie--Poisson systems subject to perturbations by transport noise. 
A unique feature of the integrator is its scalability to higher dimensions while not being limited to separable Hamiltonians, providing a major benefit embodied in computational efficiency and wider applicability. 
Furthermore, we provide a rigorous derivation of the integrator along with error estimates and convergence proofs, by combining computational geometric mechanics with stochastic analysis.
LP systems arise through symmetry reduction of Hamiltonian systems and are further detailed below. One of themathematical structures of these systems manifests itself in conserved quantities referred to as Casimir functions.  
The main proven results are informally summarized as follows. 
\begin{itemize}
    \item The integrator is a \emph{Lie--Poisson integrator} and preserves Casimir functions. 
    \item The integrator converges with root mean squared order $1/2$ and weak order $1$. 
\end{itemize}
Specifically, the integrator is designed for LP systems that admit identification with isospectral formulations. This is the case when the underlying Lie groups are $J$-quadratic, as shown in Section \ref{sec:geometry}. The identification with isospectral flows enables efficient numerical algorithms for the solution of stochastic LP equations. In particular, we develop an integrator based on a stochastic implicit midpoint scheme for suitable isospectral Hamiltonian flow and free of computationally expensive algebra-to-group maps. This integrator can be applied to stochastic canonical low-dimensional examples such as the rigid body and point vortex dynamics, as shown in \Cref{fig:eyecandy_1}. 
A major benefit of the integrator compared to those based on the exponential map or other algebra-to-group maps (see e.g., \cite{luesink2024casimir, gay2024variational}) is its scalability to higher dimensions. This is illustrated in \Cref{fig:eyecandy_2}, showing snapshots of a high-resolution numerical simulation of the stochastic two-dimensional Euler equations on the sphere, enabling the study of transport noise in turbulence \cite{ephrati2025diffusive}.
In addition, we prove convergence rates of the integrator and demonstrate convergence in numerical tests. 

Hamiltonian descriptions of conservative mechanical systems allow for coordinate-in\-de\-pen\-dent formulations. In these formulations, many canonical systems encountered in classical mechanics can be defined on the cotangent bundle of a Lie group. This means that the position variables, here denoted by $Q$, are represented as an element of a Lie group, and simultaneously the conjugate momenta, denoted by $P$, are represented as an element of the cotangent space at $Q$. In particular, if the corresponding Hamiltonian is invariant with respect to the action of the Lie group, these systems give rise to LP systems through so-called Lie--Poisson reduction, see \cite{marsden1974reduction}. Geometric integrators for LP systems typically intend to preserve the Casimir functions or, equivalently, the symplectic foliation. In the current study, we derive an integrator for a class of stochastic LP systems that preserves the Casimir functions to machine precision.

The stochastic LP systems studied in this paper are derived from stochastic Hamiltonian systems. Stochastic Hamiltonian mechanics were first introduced by Bismut \cite{bismut1982mecanique} on symplectic vector spaces. Lazaro--Cami and Ortega \cite{lazaro2007stochastic} extended these results to stochastic Hamiltonian mechanics on Poisson manifolds while using general continuous semimartingales. Furthermore, they showed that stochastic mechanics equations arise as a critical point of a stochastic action. The work by Bou--Rabee and Owhadi \cite{bou2009stochastic} considered the special cases of stochastic Hamiltonian systems driven by Wiener processes and was able to prove that an extremum of the stochastic action satisfies the stochastic Hamiltonian equations. This has spurred the research of variational integrators for stochastic canonical Hamiltonian equations \cite{bou2009stochastic}. Additional results include the development of high-order symplectic methods (e.g., \cite{deng2014high, anton2014symplectic, ma2012symplectic, wang2017construction}), energy- and drift-preserving methods \cite{cohen2014energy, chen2020drift}, and extensions to general stochastic Hamiltonian systems \cite{holm2018stochastic} and diffusive Hamiltonian systems \cite{kraus2021variational}.

Integrators for stochastic LP systems are often derived in the same manner as their deterministic counterparts. For an overview of deterministic geometric methods for LP systems, we refer to the textbook \cite{hairer2006geometric} and review papers \cite{iserles2000lie, celledoni2014introduction, marsden2001discrete}. The recent work \cite{brehier2023splitting} presented a stochastic LP integrator based on splitting techniques, yielding an efficient explicit integrator for a class of LP systems driven by Stratonovich noise. The integrator developed by \cite{luesink2024casimir} extended Runge--Kutta Munthe--Kaas (RKMK) schemes \cite{munthe1999high, engo2001numerical} to stochastic systems that have an LP Hamiltonian formulation. An alternative approach based on the Darboux--Lie theorem was presented by \cite{hong2021structure}. The approach maps the LP equations to canonical stochastic Hamiltonian systems, which are discretized via specialized symplectic schemes, after which the inverse coordinate transform is applied to obtain the LP equations. In the present study, we employ the discrete LP reduction method presented by \cite{modin2020lie} to derive an integrator for stochastic LP systems.

The presented integrator is derived using tools from geometric mechanics. Furthermore, the proofs for convergence of the integrator and the desired conservation properties leverage the geometric structure of LP systems and the integrator. We note that the results presented in this paper may readily be extended to a larger class of stochastic LP integrators based on symplectic Runge--Kutta methods, in the same fashion as higher-order deterministic LP integrators can be derived in the deterministic setting \cite{modin2020lie}. To achieve this, one can adopt symplectic integrators for the stochastic canonical Hamiltonian system, besides the implicit midpoint scheme. Examples of such schemes were developed by \cite{milstein2002numerical}. Subsequently, the process of discrete Lie--Poisson reduction described in the current paper can be repeated for the adopted symplectic integrator to derive a Lie--Poisson integrator.

The paper is structured as follows. In \Cref{sec:geometry}, we provide the necessary theory from geometric mechanics that is used in the derivation and analysis of the integrator. Subsequently, the stochastic Lie--Poisson systems are presented in \Cref{sec:stoch_lp}. The integrator for stochastic Lie--Poisson systems is derived in \Cref{sec:isomp} followed by error analysis and convergence proofs in \Cref{sec:error_analysis}. Numerical examples are provided in \Cref{sec:examples} and we conclude the paper in \Cref{sec:conclusions}.

\begin{figure}
    \centering
    \includegraphics[width=0.48\textwidth]{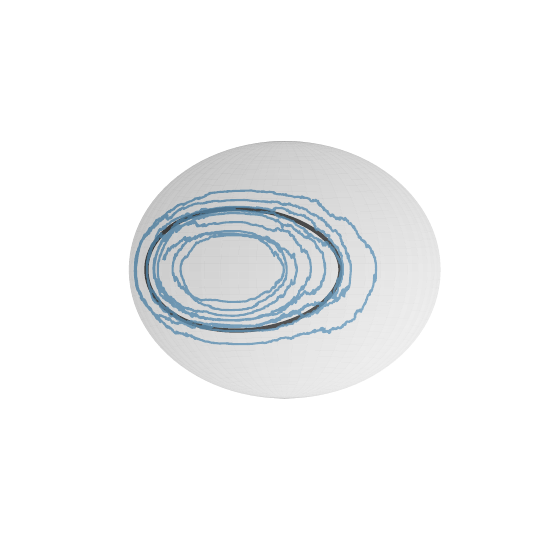}
    \includegraphics[width=0.48\textwidth]{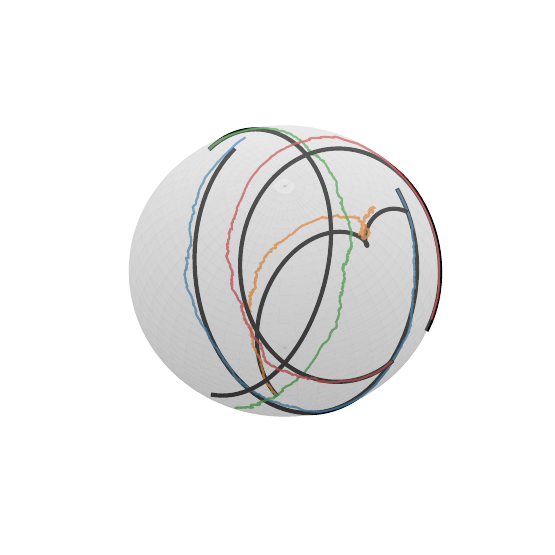}
    \vspace{-20mm}
    \caption{Left: Trajectories of components of the angular velocity in the rigid body equations. The lines show numerical solutions to the deterministic equations (black) and the stochastic equations (blue). Note that both trajectories remain on the unit sphere, depicted in gray. Right: Trajectories of four point vortices on the sphere, following the deterministic dynamics (black) and stochastic dynamics (colored).}
    \label{fig:eyecandy_1}
\end{figure}

\begin{figure}
    \centering
    \includegraphics[width=\textwidth]{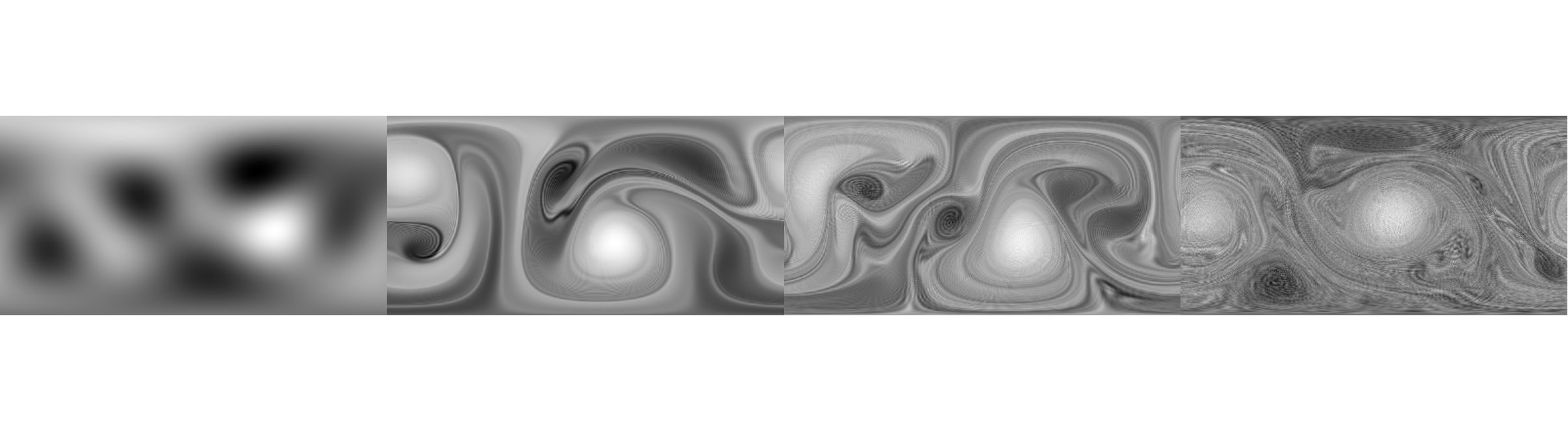}\vspace{-10mm}
    \caption{Instantaneous vorticity snapshots of the stochastic two-dimensional Euler equations on the sphere, presented on a latitude-longitude grid. The stochastic LP system evolves on $\mathfrak{su}(512)^*$. Shown are the initial condition consisting only of large-scale components (left); turbulent mixing during a transitory phase (middle left, middle right); formation of large-scale vorticity condensates (right).}
    \label{fig:eyecandy_2}
\end{figure}

\section{Geometric structure of deterministic Lie--Poisson systems}
\label{sec:geometry}
In this section, we introduce several concepts from geometric mechanics central to the derivation of the stochastic integrator in Section~\ref{sec:isomp}, and in particular, we introduce Lie--Poisson systems and their geometric structure. One may consult textbooks on geometry and mechanics \cite{frankel2011geometry,holm2009geometric, marsden2013introduction} for a more detailed theoretical background of the material presented here.

Let $G$ be a Lie group with Lie algebra $\mathfrak{g}$. 
Given a Hamiltonian function $H: \mathfrak{g}^* \to \R$, the corresponding Lie--Poisson flow on $\mathfrak{g}^*$, the dual of the Lie algebra, is
\begin{align}
\label{eq:lie_poisson_deterministic_adform}
 \dot X_t = \operatorname{ad}^*_{\nabla H(X_t)} X_t,
\end{align}
where $\nabla H$ is the gradient of $H$ and the operator $\operatorname{ad}^*:\mathfrak{g} \times \mathfrak{g}^* \to \mathfrak{g}^*$ is given by 
\begin{align}
    \label{eq:adstardef}
    \langle \operatorname{ad}_U^*X,V  \rangle =  \langle X,[U,V]  \rangle.
\end{align}
Here $X\in \mathfrak{g}^*$, $U,V \in \mathfrak{g}$, $ \langle  \cdot,\cdot\rangle$ denotes the dual pairing between $\mathfrak{g}^*$ and $\mathfrak{g}$, and $[\cdot,\cdot]$ is the Lie bracket. 

Throughout this paper, the underlying groups are compact simply connected matrix Lie groups $G\subseteq \operatorname{GL}(n, \C)$ whose Lie algebras $\mathfrak{g}\subseteq \mathfrak{gl}(n, \C)$ are assumed to be $J$-quadratic. 
Let $A^*$ denote the conjugate transpose of a matrix $A \in \C^{n\times n}$. 
A matrix Lie algebra $\mathfrak{g}$ is $J$-quadratic if there exists a matrix $J$ such that all $A\in\mathfrak{g}$ satisfy 
\begin{equation}
    \label{eq:j-quad}
    A^* J + JA = 0
\end{equation}
with $J^* =\pm J $ and $J^2 = cI_n$, for $c\in\R$ non-zero, where $I_n$ denotes the $n\times n$ identity matrix.
An important fact for our purposes is that $J$-quadratic Lie algebras are reductive. This implies that the Lie algebra is closed under conjugate transpose. Indeed, if $A \in \mathfrak{g}$, then $A^*$ satisfies 
\begin{align*}
    AJ + JA^* = \pm(JA^* + AJ)^* = 0.  
\end{align*}
Note that \Cref{eq:j-quad} implies that $G$ is defined by the quadratic constraint 
\begin{equation}
    \label{eq:quadconst}
    Q \in G \iff Q^* JQ = J. 
\end{equation}

Examples of $J$-quadratic Lie algebras include $\mathfrak{su}(n)$, $\mathfrak{sp}(n)$ and $\mathfrak{so}(n)$, respectively, consisting of traceless skew-Hermitian $n\times n$ matrices, $2n\times 2n$ matrices $A$ that satisfy $A^*J + JA = 0$ with $J$ the symplectic matrix, and skew-symmetric $n\times n$ matrices. The Lie algebra $\mathfrak{g}$ is identified with its dual $\mathfrak{g}^*$ through the pairing $\langle\cdot, \cdot\rangle: \mathfrak{g}^* \times \mathfrak{g} \to \R$. 
We adopt the Frobenius inner product 
\begin{equation}
    \label{eq:frobeniusip}
    \langle X, V \rangle = \trace (X^* V)
\end{equation}
as the pairing, where $X^*$ indicates the conjugate transpose of $X$. 
Since $G$ is assumed compact, a bi-invariant metric exists, allowing the identification of $\mathfrak{g}$ with $\mathfrak{g}^*$.
The norm associated to the Frobenius inner product is denoted by $\|\cdot\|$.

One major reason for these assumptions is that if the Lie algebra is closed under conjugate transpose, as is the case for reductive Lie algebras, the Lie--Poisson equations are \emph{isospectral}.
Indeed, combining \Cref{eq:frobeniusip} and \Cref{eq:adstardef}, we have that $
 \operatorname{ad}_U^* X=  \Pi[U^*, X] $, 
where $\Pi:\mathfrak{gl}(N,\C) \to \mathfrak{g}$ is the orthogonal projection from all matrices down to the Lie algebra. 
If, however, $X^* \in \mathfrak{g}$ for all $X \in \mathfrak{g}$, i.e., the Lie algebra is closed under conjugate transpose, then $\Pi[X^*, V] = [X^*, V]$ and the system \eqref{eq:lie_poisson_deterministic_adform} becomes the isospectral equation 
\begin{equation}
\label{eq:lie_poisson_deterministic}
    \dot X_t = [\nabla H(X_t)^*,X_t], 
\end{equation}
where the gradient of $H$ is taken with respect to the adopted inner product. Note that here the assumption that the algebra is $J$-quadratic is essential, as it guarantees that $\nabla H(X)^*$ remains in the (dual of) the algebra, meaning that \Cref{eq:lie_poisson_deterministic} is a Lie--Poisson system. 

By assuming that $G$ is compact and simply connected, we can formulate concise proofs of the theoretical results in this paper.

In what follows, the cotangent bundle of a Lie group is denoted by $T^* G$ with elements~$(Q, P)$. The left action of $G$ on $T^* G$ is given by \begin{equation}
    g\cdot (Q, P) = (gQ, (g^{-1})^*P), \quad g\in G.
\end{equation}
In many mechanical systems, this action $\Phi\colon G\times T^*G\to T^*G$ is the cotangent lift of the group action and is often generated by a Hamiltonian $\tilde{H}\colon T^*G\to\R$. If this Hamiltonian is itself invariant under the left or right multiplication by an element of the Lie group, then symmetry reduction is possible \cite{marsden1974reduction}. The result of this symmetry reduction is that the equations of motion associated with the action $\Phi$ can be expressed on~$\mathfrak{g}^*$. 

We consider Hamiltonian functions $\tilde{H}\colon T^*G \to \R$ that are invariant under the left actions, i.e., $\tilde{H}(Q, P) = \tilde{H}(g\cdot (Q, P))$
and can therefore be given by the reduced Hamiltonian $H\colon\mathfrak{g}^*\to \R$ using $H(\mu(Q, P)) = \tilde{H}(e, \mu(Q, P)) = \tilde{H}(Q, P)$. Here, $e$ is the identity element of the group and $\mu\colon T^* G \to \mathfrak{g}^*$ is the momentum map on the considered cotangent bundles, given by  \begin{equation}
    \mu(Q, P) = \frac{1}{2}Q^* P - \frac{1}{2c}JP^* Q J.
    \label{eq:momentum_map}
\end{equation}
By applying the momentum map $\mu$ to elements of $T^*G$, we formally achieve a process of  Lie--Poisson reduction.

The evolution of the momentum map is given by the Lie--Poisson system \eqref{eq:lie_poisson_deterministic} formulated on the dual of the Lie algebra where $X_t=\mu(Q_t, P_t)$. The canonical Hamiltonian system corresponding to $\tilde H$ may be obtained from the Lie--Poisson system \eqref{eq:lie_poisson_deterministic} via the reconstruction equations \cite[Proposition 9.18]{holm2009geometric} \begin{equation}
    \begin{split}
        \dot{Q}_t &= Q_t\nabla H(X_t), \\
        P_t &= (Q_t^{-1})^* X_t.
    \end{split}
    \label{eq:reconstruction}
\end{equation} 

The solution $X_t$ of \eqref{eq:lie_poisson_deterministic} lives on a coadjoint orbit, denoted by $\mathcal{O}_{X_0}$ and defined by the initial condition $X_0$ as \begin{equation}
    \mathcal{O}_{X_0} = \left\{ g^* X_0 (g^*)^{-1} | g\in G \right\}.
    \label{eq:coadjoint_orbit}
\end{equation}
The motion on the coadjoint orbit is generated by the Lie--Poisson system $\eqref{eq:lie_poisson_deterministic}$, i.e., $g$ in \eqref{eq:coadjoint_orbit} is time-dependent and solves the differential equation \begin{equation}
    \dot g_t =  g_t \nabla H(X_t) .
    \label{eq:generator_deterministic}
\end{equation}

By inserting \eqref{eq:lie_poisson_deterministic} and the momentum map \eqref{eq:momentum_map} into the reconstruction equations \eqref{eq:reconstruction}, the canonical Hamiltonian system on the cotangent bundle of quadratic Lie groups is obtained. The system is given by 
\begin{equation}
    \begin{split}
        \dot{Q}_t &=Q_t\nabla H(\mu(Q_t,P_t)), \\ 
        \dot{P}_t &= -P_t\nabla H(\mu(Q_t,P_t))^*.
    \end{split}
    \label{eq:canonical_hamiltonian_deterministic}
\end{equation}
The equation of this system can also be written in terms of $\tilde{H}$ using the invariance under the group action. 
\begin{remark}
    \label{rem:canonical_hamiltonian_deterministic}
    The canonical Hamiltonian system \eqref{eq:canonical_hamiltonian_deterministic} is a system on the vector space $\R^{n\times n} \times \R^{n \times n}$, of which $T^*G$ is a submanifold. 
    However, the reconstruction procedure of \cite[Proposition 9.18]{holm2009geometric} ascertains that $(Q_t,P_t) \in T^*G$. 
    In particular, this means that $Q_t^* J Q_t$ is a quadratic invariant of the canonical Hamiltonian system \eqref{eq:canonical_hamiltonian_deterministic}. 
\end{remark}

Lie--Poisson systems have conserved quantities known as \emph{Casimirs}, independent of the choice of the Hamiltonian function. 
These are functions $C:\mathfrak{g}^* \to \R$ that are constant on coadjoint orbits. These functions and the coadjoint orbits of the system are important for the long-term behavior of Lie--Poisson systems, and their preservation is an important goal when constructing numerical integration schemes. So-called \emph{Lie--Poisson integrators} preserve the coadjoint orbits, Casimir functions, and the symplectic structure on the coadjoint orbits \cite{arnaudon2018noise, brehier2023splitting}. 
The meaning of preserving the symplectic structure on the coadjoint orbits requires that we introduce the \emph{Poisson bracket} on $C^\infty(\mathfrak{g}^*)$. 

To introduce the Poisson bracket, let $\varrho \in \mathfrak{g}^*$ 
and let $f \in C^\infty(\mathfrak{g}^*)$. 
Further, let $\frac{\delta f}{\delta \varrho}\in\mathfrak{g}$ denote the variational derivative of $f$ with respect to $\varrho$.
In the finite-dimensional setting, we can interpret variational derivatives as partial derivatives.
Recall that the Lie algebra $\mathfrak{g}$ can be identified with its dual $\mathfrak{g}^*$ via the dual pairing. 
Let $D = \operatorname{dim}(\mathfrak{g}^*)$. A basis $(e_i)_{i=1}^D$ of $\mathfrak{g}$ induces a dual basis $(\varepsilon_i)_{i=1}^D$ on $\mathfrak{g}^*$, so that any $\varrho \in \mathfrak{g}^*$ can be written as
\begin{align*}
    \varrho = \sum_{i=1}^D \varrho_i \varepsilon_i.
\end{align*}
The Poisson bracket on $C^\infty(\mathfrak{g}^*)$ is now given by 
\begin{align*}
    \{f,g\}(\varrho) = \sum_{i,j,k = 1}^{D} C_{ij}^k \varrho_k \frac{\partial f}{\partial \varrho_i}\frac{\partial g}{\partial  \varrho_j},
\end{align*}
where we understand $\varrho \in \mathfrak{g}^*$ as a vector and  $C_{ij}^k$ are the structure constants of $\mathfrak{g}$. 

Let $I\colon \mathfrak{g}^* \to \R$ be a smooth map. Then $I$ evolves along a solution $X_t$ of the system~\eqref{eq:lie_poisson_deterministic} by the equation 
\begin{align}
\dot I(X_t) = \{I,H\}(X_t).
\end{align}
If it holds that $\{I,H \} = 0$, then $I$ is a conserved quantity, i.e., $\dot I(X_t) = 0$. 
An important class of conserved quantities are the functions $C:\mathfrak{g}^* \to \R$ that Poisson commute with every other function.
These are precisely the Casimirs. 
The Casimirs depend solely on the Poisson structure, and not on the choice of Hamiltonian.

Finally, the Poisson property of a mapping is defined using the Poisson bracket. A mapping $\phi: \mathfrak{g}^* \to \mathfrak{g}^*$ is a \emph{Poisson map} if for all $f,g \in C^\infty(\mathfrak{g}^*)$,
\begin{align*}
\{f\circ \phi,g\circ \phi\} = \{f,g\} \circ \phi,
\end{align*}
i.e., the \emph{pullback} of $\phi$ preserves the Poisson bracket, meaning that the symplectic structure on each coadjoint orbit is preserved. 

\section{Stochastic Lie--Poisson systems}
\label{sec:stoch_lp}

We now turn our attention to stochastic Lie--Poisson systems. Throughout this section, let  $W^1, \ldots, W^M$ be $M$ independent Brownian motions defined with respect to $(\Omega,\mathcal{F},(\mathcal{F}_t)_{t\geq 0},\mathbb{P})$, where $(\Omega,\mathcal{F},\mathbb{P})$ is a complete probability space and $(\mathcal{F}_t)_{t\geq 0}$ is the filtration. Typically, $M$ is chosen to be equal to the dimension of the underlying space so that the noise is isotropic \cite{arnaudon2018noise}.

To arrive at stochastic Lie--Poisson systems with solutions defined on the same coadjoint orbit \eqref{eq:coadjoint_orbit} and therefore the same conserved quantities as the deterministic system \eqref{eq:lie_poisson_deterministic}, the differential equation \eqref{eq:generator_deterministic} generating the group element is made stochastic. We denote the original Hamiltonian function by $H_0$ and introduce $M$ noise Hamiltonians $H_k\colon\mathfrak{g}^*\to \R$, $k=1,\ldots, M$.  The generating function then reads \begin{equation}
    \dd g_t = \nabla H_0(X_t)\, \dd t + \sum_{k=1}^M \nabla H_k(X_t)\circ\dd W_t^k,
    \label{eq:generator_stochastic}
\end{equation}
and the corresponding stochastic Lie--Poisson system in Stratonovich form is given by 
\begin{align}
\label{eq:sys_strat}
	 \dd X_t& =[\nabla H_0(X_t)^*,X_t] \, \dd t	+ \sum_{k=1}^M [\nabla H_k(X_t)^*,X_t] \circ \dd W_t^k,
\end{align} 
with initial condition~$X_0\in \mathfrak{g}^*$. 
This type of stochastic Lie--Poisson system has been the topic of previous studies \cite{brehier2023splitting, luesink2024casimir, arnaudon2018noise}.
The Stratonovich integral is used in the formulation \eqref{eq:generator_stochastic} and consequently in \eqref{eq:sys_strat}, since it allows for extending manifold-valued curves in differential geometry to manifold-valued processes using Stratonovich calculus \cite{emery2006two}, as a result of its convenient property that the ordinary chain rule holds for Stratonovich processes \cite{kloeden1992stochastic}.
 
The It\^{o} form can also be adopted but requires a choice of connection \cite{emery2006two, huang2023second} and is not further considered here.

To address local and global existence and uniqueness of solutions to the system \eqref{eq:sys_strat}, one could assume sufficiently regular coefficients to apply standard existence results in the literature, see e.g., \cite[Section 5.3]{protter2005stochastic}.  However, many of the assumptions necessary for global existence and uniqueness do not hold for the stochastic Lie--Poisson systems that we intend to study. 
For instance, we typically cannot assume global Lipschitz continuity.
Nevertheless, the global existence and uniqueness of solutions to \eqref{eq:sys_strat} was proven under mild conditions in \cite[Proposition 1]{brehier2023splitting}, relying on aspects of the underlying geometric structure. 
In particular, they utilize in their proof the conservation properties of the Casimirs, i.e., analytic real-valued functions constant on the coadjoint orbits \cite{arnaudon2018noise}.
Alternatively, Casimirs can be defined as the elements of the kernel of the Poisson bracket on $C^\infty(\mathfrak{g}^*)$. 
More specifically, if $I: \mathfrak{g^*} \to \R$ is a smooth map, then $I \circ X$ evolves (as a random quantity) along a solution $X$ of the system \eqref{eq:sys_strat} according to the equation  
\begin{align}
\dd I(X_t) = \{I,H_0\}(X_t)\, \dd t + \sum_{i=1}^M\{I,H_i \}(X_t) \circ \dd W_t^i. 
\end{align}
If $\{I,H_i \} = 0$ for all $i = 0,1,\dots, M$, then $I$ is a conserved quantity, i.e., $\text{d}I(X_t) = 0$. 
If  $C:\mathfrak{g^*} \to \R$ belongs to the kernel of the Poisson bracket, it is a conserved quantity and the kernel coincides with the Casimirs. 
In particular, this means that along the flow of the stochastic Poisson system, $\text{d}C(X_t) = 0$ everywhere. 

In \cite[Proposition 1]{brehier2023splitting}, it is shown that if the stochastic Lie--Poisson system has a Casimir function~$C$ that has compact level sets, i.e., the sets 
\begin{align*}
    \{X \in \mathfrak{g^*}: C(X) = c\}
\end{align*}
are compact for all $c \in \R$, then the stochastic Lie--Poisson system admits a unique global solution.
This is the case for system \eqref{eq:sys_strat} as it evolves on the dual of a $J$-quadratic Lie algebra of a compact Lie group, so \cite[Proposition 1]{brehier2023splitting} ensures global existence and uniqueness of solutions to system \eqref{eq:sys_strat}.

A central property of the solution of a stochastic Lie--Poisson system is that its flow preserves the coadjoint orbits, see, e.g., \cite[Theorem 2.8]{brehier2023splitting}. Preserving this distinctive feature numerically is also pursued in our work below by constructing a Casimir preserving Lie--Poisson integrator. 
Namely, we develop a numerical integrator $\Psi_h\colon\mathfrak{g}^*\times [0, T]\times \Omega \to \mathfrak{g}^* $ to the system \eqref{eq:sys_strat}, that for all step sizes $h>0$: 
\begin{enumerate}
    \item is almost surely a Poisson map.
    \item is almost surely a Lie--Poisson mapping. That is, the numerical flow almost surely preserves the coadjoint orbits and the symplectic structure on the coadjoint orbits, and almost surely preserves Casimir functions.  
\end{enumerate}

\begin{remark}
    \label{rem:casimir_is_ev_pres}
    The systems \eqref{eq:lie_poisson_deterministic} and \eqref{eq:sys_strat} describe isospectral flows. By definition, the eigenvalues, and the trace of $X_t$ are therefore conserved. Casimir preservation is implied by eigenvalue conservation \cite{modin2020lie}, namely, for any analytic function $f$ extended to matrices we may define a Casimir $C_f:\mathfrak{g}^* \to \R$ by 
    \begin{align*}
        C_f(X_t) = \operatorname{Tr}(f(X_t)).
    \end{align*}
\end{remark}

\section{Derivation of the isospectral midpoint method}
\label{sec:isomp}

Central to the derivation of the numerical method is stochastic Lie--Poisson reduction theory. 
The flow of the system~\eqref{eq:sys_strat} is the Poisson reduction of a canonical stochastic Hamiltonian system on~$T^*G$, or equivalently, the canonical stochastic Hamiltonian system is reconstructed from the stochastic Lie--Poisson system via the reconstruction equations \eqref{eq:reconstruction}, following the same approach as for the deterministic equations. 
By applying \Cref{eq:reconstruction} to \Cref{eq:sys_strat} we obtain the canonical stochastic Hamiltonian system on $T^*G$ given by
\begin{equation}
    \label{eq:hamsystem}
    \begin{split}
        \dd Q_t &= Q_t\nabla H_0(\mu(Q_t,P_t)) \, \dd t + \sum_{i=1}^M Q_t \nabla H_i (\mu(Q_t,P_t))\circ\dd W_t^i, \\
        \dd P_t &= -P_t\nabla H_0(\mu(Q_t,P_t))^* \, \dd t - \sum_{i=1}^M P_t \nabla H_i (\mu(Q_t,P_t))^* \circ\dd W_t^i.
    \end{split}
\end{equation}
\begin{remark}
\label{rem:canonical_hamiltonian_stochastic}
It is important to note that there are no a priori guarantees that the  flow of the stochastic Hamiltonian system \eqref{eq:hamsystem} remains on $T^*G$.
Indeed, by writing the equation in the form \eqref{eq:hamsystem}, we are implicitly embedding the system into the vector space $T^* R^{n \times n} = \R^{n\times n} \times \R^{n\times n} $, just as in the deterministic case, see \Cref{rem:canonical_hamiltonian_deterministic}. 
However, the reconstruction equations \cite[Proposition 9.18]{holm2009geometric}, also outlined in \Cref{app:reconstruction_stoch_hamiltonian}, ascertain that $(Q_t,P_t) \in T^*G$ for all times and all noise realizations. 
In particular, since $Q_t \in G$, it holds that $Q_t J Q_t = J$, meaning that $Q_t J Q_t$ is a quadratic invariant of the system. Furthermore, this means that we can appeal to results on vector spaces when proving results about the stochastic Hamiltonian system \eqref{eq:hamsystem}.

\end{remark}

As we shall see in Sections \ref{subsec:stochastic_midpoint} and \ref{subsec:reduction_midpoint}, a stochastic Lie--Poisson integrator can under certain conditions be obtained by applying the momentum map to an integrator of the stochastic Hamiltonian system \eqref{eq:sys_strat}. This yields a Lie--Poisson integrator if the integrator for the corresponding stochastic Hamiltonian system is a symplectic integrator and equivariant with respect to the group action generating the dynamics \cite{modin2020lie}. We elaborate on these technical aspects next.

\subsection{Existence and uniqueness of solutions to the canonical stochastic Hamiltonian system.}

The auxiliary canonical stochastic Hamiltonian system \eqref{eq:hamsystem} on $T^* G$ is used to derive the integration method for the system \eqref{eq:sys_strat} on $\mathfrak{g}^*$. Before proceeding to construct the integrator for the Lie--Poisson system \eqref{eq:sys_strat}, we must ensure the global existence and uniqueness of solutions to the auxiliary stochastic Hamiltonian system \eqref{eq:hamsystem}.

\begin{lemma}
\label{lem:upstairs_existence}
Let $H_0$ be of class $C^1$ and let $H_1,H_2, \ldots, H_M$ be of class $C^2$. 
Then, for any deterministic initial condition $(Q_0,P_0) \in T^*G$, \Cref{eq:hamsystem} has a unique global solution  $(Q_t,P_t)_{t \geq 0}$, with $(Q(0),P(0))= (Q_0,P_0)$. 
Further, there is a constant $R(Q_0,P_0)$ depending on the initial value $(Q_0,P_0)$ such that 
\begin{align*}
    \|(Q_t,P_t)\| \leq R(Q_0,P_0). 
\end{align*}
\end{lemma}
\begin{proof}
    We follow the idea of the proof of \cite[Proposition 2.2]{brehier2023splitting}. 
    Note that in our setting, the proof cannot be directly applied, as the present paper considers a Hamiltonian system on the cotangent bundle of a Lie group and not a Lie--Poisson system on a vector space. 
    
    Let $C:\mathfrak{g}^* \to \R$ be a Casimir with compact level sets. Such a Casimir can always be found since the group is assumed to be compact.
    The stochastic Lie--Poisson system \eqref{eq:sys_strat} evolves on coadjoint orbits, which are by assumption compact. 
    The canonical stochastic Hamiltonian system \eqref{eq:hamsystem} reduces to the stochastic Lie--Poisson system \eqref{eq:sys_strat} via the momentum map $\mu$.
    Thus, the inverse image of a coadjoint orbit consists of all the $\mu$-fibers in $T^*G$ with the coadjoint orbit as base space. 
    The inverse image is an orbit in $T^*G$, and as $G$ is compact, the orbit is compact. Each $(Q_0,P_0)$ belongs to such a compact set, which we denote $O_{(Q_0,P_0)}$.
    Let 
    \begin{align*}
         R(Q_0,P_0) = \max_{\substack{(Q,P) \in O_{(Q_0,P_0)} | C(\mu(Q,P)) = C(\mu(Q_0,P_0)}} \|Q_0,P_0\|.
    \end{align*} 
    The maximum is finite, since it is taken over a compact set.
    Note that $ R(Q_0,P_0)$ is a deterministic quantity, depending solely on the deterministic initial conditions and the Casimir. 
    Let $R = R(Q_0,P_0) +1 $. 
    Define the truncated Hamiltonians $T_i:\mathfrak{g}^* \to \R$, $i = 1,\dots, M$, as the mappings that have compact support in the ball $B_R = \{X \in \mathfrak{g}^*\colon \|X\|\leq R\}$ and that coincide with $H_i$ in the ball $\{X \in \mathfrak{g}^*: \|X\|\leq R(Q_0,P_0) \}$. 
    Consider the corresponding stochastic Lie--Poisson system
    \begin{align}
    \label{eq:sys_strat_trunc}
	\begin{split}
	 &\dd X^R_t=[\nabla T_0(X^R_t)^*,X^R_t]\, \dd t	+ \sum_{i=1}^M [\nabla  T_i(X^R_t)^*,X^R_t] \circ \dd W_t^k.
	\end{split}
\end{align}

We note that due to the compact support of $T_0$, $[\nabla T_0(X^R_t)^*,X^R_t]$ is globally Lipschitz (since it is given by the product of Lipschitz functions and has compact support) and  $[\nabla  T_k(X_t)^*,X_t] $ is bounded with bounded derivatives for all $k = 1,\dots, M$. 
The canonical stochastic Hamiltonian system corresponding to \Cref{eq:sys_strat_trunc} is given by 
\begin{equation}
    \label{eq:hamsystem_trunc}
    \begin{split}
        \dd Q^R_t &= Q^R_t\nabla T_0(\mu(Q^R_t,P^R_t)) \, \dd t + \sum_{i=1}^M Q^R_t \nabla T_i (\mu(Q^R_t,P^R_t))\circ\dd W_t^i, \\
        \dd P^R_t &= -P^R_t\nabla T_0(\mu(Q^R_t,P^R_t))^* \, \dd t - \sum_{i=1}^M P^R_t \nabla T_i (\mu(Q^R_t,P^R_t))^* \circ\dd W_t^i,\\
    \end{split}
\end{equation}
with initial value $Q^R_0 =Q_0, P^R_0 =P_0$.
Note that since $T_0\colon \mathfrak{g}^* \to \R$ has compact support, so does $T_0 \circ \mu$. 
Indeed, the support of $T_0 \circ \mu$ consists of all $(Q,P) \in T^*G$ such that $\mu(Q,P) \in B_R$, i.e., the collection of $\mu$-fibers in $T^*G$ with $B_R$ as base space.
This is an orbit in $T^*G$, and as $G$ is compact, the orbit is compact. 
Therefore, the support of $T_0 \circ \mu$ is compact, and we conclude that the drift coefficients of \Cref{eq:hamsystem_trunc} are globally Lipschitz and the diffusion coefficients are bounded with bounded derivatives. 
Standard existence results now give a unique global solution $(Q^R_t,P^R_t)_{t \geq 0}$ to \Cref{eq:hamsystem_trunc} (see, e.g., \cite{Milstein2004}).

Furthermore, as \Cref{eq:hamsystem_trunc} reduces to \Cref{eq:sys_strat_trunc}, it holds that $C(\mu(Q^R_t,P^R_t)) = C(\mu(Q_0,P_0)) $ for all $t$. 
By definition, this means that $\|(Q^R_t,P^R_t)\| \leq R(Q_0,P_0)$ or in other words, that the solution $(Q^R_t,P^R_t)$ remains in the inverse image of the ball with radius $R(Q_0,P_0)$, so that $T_i(\mu(Q^R_t,P^R_t)) = H_i(\mu(Q^R_t,P^R_t)$. 
Thus, it holds that 
\begin{equation}
    \label{eq:hamsystem_trunc_2}
    \begin{split}
        \dd Q^R_t &= Q^R_t\nabla H_0(\mu(Q^R_t,P^R_t)) \, \dd t + \sum_{i=1}^M Q^R_t \nabla H_i (\mu(Q^R_t,P^R_t))\circ\dd W_t^i, \\
        \dd P^R_t &= -P^R_t\nabla H_0(\mu(Q^R_t,P^R_t))^* \, \dd t - \sum_{i=1}^M P^R_t \nabla H_i (\mu(Q^R_t,P^R_t))^* \circ\dd W_t^i,\\
    \end{split}
\end{equation}
with initial value $Q_0^R = Q_0,P_0^R = P_0$. 
Since \Cref{eq:hamsystem_trunc_2} coincides with \Cref{eq:hamsystem}, the theorem follows by noting that uniqueness follows from the local well-posedness of \Cref{eq:hamsystem}. 
\end{proof}

\subsection{The stochastic implicit midpoint method}\label{subsec:stochastic_midpoint}
We adopt the stochastic symplectic implicit midpoint method \cite{milstein2002numerical} for the Hamiltonian system \eqref{eq:hamsystem} and denote this integrator by $\Phi_h: T^*G \times [0, T] \times \Omega \to T^*G$, 
where $h$ is the time step size. The explicit dependence on $t\in [0, T]$ and $\omega\in\Omega$ will be omitted in what follows. We introduce the notation \begin{equation}
    \begin{split}
        f_i(Q, P) &= \frac{1}{2}Q\nabla H_i(\mu(Q, P)), \\
        k_i(Q, P) &= -\frac{1}{2}P\nabla H_i(\mu(Q, P))^*,
    \end{split}
\end{equation}
for $i=0, \ldots, M$. The integrator then reads as the composition $\Phi_h = \Phi_h^{(2)} \circ \Phi_h^{(1)}$ of two steps given by \begin{equation}
    \label{eq:stoch_midpoint_hamiltonian}
    \begin{split}
    \Phi_h^{(1)} &\colon \begin{cases}
        Q_n & = \tilde{Q}- \frac{1}{2}\left( f_0(\tilde{Q}, \tilde{P}) h + \sum_{i=1}^M f_i(\tilde{Q}, \tilde{P})(\zeta_{i})_n \sqrt{h}\right), \\
        P_n & = \tilde{P}-\frac{1}{2}\left(k_0(\tilde{Q}, \tilde{P}) h + \sum_{i=1}^M k_i(\tilde{Q}, \tilde{P})(\zeta_{i})_n\sqrt{h} \right),
    \end{cases} \\
    \Phi_h^{(2)} &\colon \begin{cases} Q_{n+1} & = \tilde{Q} + \frac{1}{2}\left( f_0(\tilde{Q}, \tilde{P}) h + \sum_{i=1}^M f_i(\tilde{Q}, \tilde{P})(\zeta_{i})_n \sqrt{h}\right), \\
    P_{n+1} & = \tilde{P}+\frac{1}{2}\left(k_0(\tilde{Q}, \tilde{P}) h + \sum_{i=1}^M k_i(\tilde{Q}, \tilde{P})(\zeta_{i})_n\sqrt{h} \right),
    \end{cases}
    \end{split}
\end{equation}
where the subscripts $n, n+1$ denote the time instances. 
We denote by $(\zeta_i)_n\sqrt{h}$ a truncated random variable to accommodate implicit integration \cite{milstein2002numerical}, which is defined as follows. We let $\xi\sqrt{h}=W_i(t_n+h)-W_i(t_n)$ be a Wiener increment, with $\xi\sim\mathcal{N}(0, 1)$. The truncated variable is then given by 
\begin{align*}
    \zeta_h = \begin{cases}
        & A_h \quad \text{if } ~ \xi>A_h, \\
        &-A_h \quad \text{if } ~ \xi<-A_h,\\
        & \xi \quad \text{if} ~ |\xi| \leq A_h,
    \end{cases}
\end{align*} where $A_h = \sqrt{2l |\log h| }$ and $l$ is an arbitrary fixed positive integer. In the remainder of the discussion of the stochastic implicit midpoint method, the subscript $h$ is omitted from the truncated random variable. Using the truncated variables additionally allows us to invert the implicit step of the midpoint method, which is required in the proof of Lemma \ref{lem:equivariance} as well as when proving convergence. 

\begin{remark}
It is important to note that for general Lie groups~$G$ there is no a priori guarantee that the numerical method \eqref{eq:stoch_midpoint_hamiltonian} remains on $T^*G$. 
However, the group $G$ is defined by the quadratic invariant \eqref{eq:quadconst}, and it was proven in \cite{Abdulle2012} that quadratic invariants are exactly preserved along the numerical flow of the implicit midpoint method, for all noise realizations. 
Therefore, $(Q_n, P_n)\in T^*G$ implies that $(Q_{n+1}, P_{n+1})=\Phi_h(Q_n, P_n)\in T^*G$.  
\end{remark}

The integrator is said to be \emph{equivariant} when $\Phi_h(g\cdot(Q, P)) = g\cdot\Phi_h(Q, P)$ for $g\in G$. 
We have the following lemma. 
\begin{lemma}
The stochastic implicit midpoint integrator \eqref{eq:stoch_midpoint_hamiltonian} is almost surely equivariant. 
    \label{lem:equivariance}
\end{lemma}

\begin{proof}
Firstly, observe that equivariance of both $\Phi_h^{(1)}$ and $\Phi_h^{(2)}$ implies equivariance of the composition $\Phi_h^{(2)}\circ \Phi_h^{(1)} = \Phi_h$. 
Secondly, the inverse $(\Phi_h^{(1)})^{-1}$ of $\Phi_h^{(1)}$ is given by an explicit formula, for which equivariance is shown below.
Note that the equivariance of $(\Phi_h^{(1)})^{-1}$ indicates equivariance of $\Phi_h^{(1)}$ since it then holds true that 
\begin{equation}
\begin{split}
    \Phi_h^{(1)}(g\cdot(Q, P)) &= \Phi_h^{(1)}(g\cdot (\Phi_h^{(1)})^{-1}(\Phi_h^{(1)}(Q, P))) \\ &= \Phi_h^{(1)}((\Phi_h^{(1)})^{-1}(g\cdot \Phi_h^{(1)}(Q, P))) = g\cdot\Phi_h^{(1)}(Q, P).
\end{split}
\end{equation}
Furthermore, the momentum map $\mu(Q, P)$ is invariant under the group action. Thus, it is straightforward to show that \begin{equation}
\begin{split}
    f_i(g\cdot(Q, P)) &= gQ\nabla H_i(\mu(gQ, (g^*)^{-1}P))= g \left(Q\nabla H_i\mu(Q, P)\right) =  g f_i(Q, P),
\end{split}
\end{equation}
and $k_i(g\cdot(Q, P)) = (g^*)^{-1}k_i(Q, P)$ follows similarly. The equivariance of $(\Phi_h^{(1)})^{-1}$ then follows from \begin{equation}
\begin{split}
    (\Phi_h^{(1)})^{-1}(g\cdot ( Q,  P)) &= \begin{pmatrix}
         gQ- \frac{1}{2}\left( f_0(gQ, (g^*)^{-1}P) h + \sum_{i=1}^M f_i(gQ, (g^*)^{-1}P)(\zeta_{i})_n \sqrt{h}\right) \\
        (g^*)^{-1}P-\frac{1}{2}\left(k_0(gQ, (g^*)^{-1}P) h + \sum_{i=1}^M k_i(gQ, (g^*)^{-1}P)(\zeta_{i})_n\sqrt{h} \right)
    \end{pmatrix}^T \\
    & = \begin{pmatrix}
         gQ- \frac{1}{2}\left( g f_0(Q, P) h + \sum_{i=1}^M gf_i(Q, P)(\zeta_{i})_n \sqrt{h}\right) \\
        (g^*)^{-1}P-\frac{1}{2}\left((g^*)^{-1} k_0(Q, P) h + \sum_{i=1}^M (g^*)^{-1} k_i(Q, P)(\zeta_{i})_n\sqrt{h} \right)
    \end{pmatrix}^T \\
    & = g\cdot\begin{pmatrix}
        Q- \frac{1}{2}\left( f_0(Q, P) h + \sum_{i=1}^M f_i(Q, P)(\zeta_{i})_n \sqrt{h}\right) \\
        P-\frac{1}{2}\left(k_0(Q, P) h + \sum_{i=1}^M k_i(Q, P)(\zeta_{i})_n\sqrt{h} \right)
    \end{pmatrix}^T \\
    &= g\cdot (\Phi_h^{(1)})^{-1}( Q,  P).
\end{split}
\end{equation}
Equivariance of $\Phi_h^{(2)}$ is shown with a similar computation. The equivariance of $\Phi_h$ then follows, since it is a composition of equivariant maps. 
\end{proof}

\subsection{Reduction of the midpoint method to a stochastic Lie--Poisson method} \label{subsec:reduction_midpoint}
The stochastic Lie--Poisson integrator is obtained by discretizing the stochastic Hamiltonian scheme with the symplectic implicit midpoint method \eqref{eq:stoch_midpoint_hamiltonian} and subsequently applying the momentum map to the discretized system. For that purpose, we adopt the notation $X_n = \mu(Q_n, P_n)$, $\tilde{X} = \mu(\tilde{Q}, \tilde{P})$, and $X_{n+1} = \mu(Q_{n+1}, P_{n+1})$. Substituting the definitions of $Q_n, Q_{n+1}, P_{n}$, and $P_{n+1}$ from \eqref{eq:stoch_midpoint_hamiltonian} into the momentum map~$\mu$ in~\eqref{eq:momentum_map}, the reduced stochastic implicit midpoint scheme takes the form 
\begin{equation}
\label{eq:stochastic_midpoint_lp}
    \begin{split}
        \tilde{\Psi}_{h, n}(\tilde{X}) &= \nabla H_0(\tilde{X})^*h + \sum_{i=1}^M\nabla H_i(\tilde{X})^*(\zeta_i)_n\sqrt{h}, \\
        X_n &= \left(I- \frac{1}{2}\tilde{\Psi}_{h, n}(\tilde{X})\right) \tilde{X} \left(I+\frac{1}{2}\tilde{\Psi}_{h, n}(\tilde{X})\right), \\
        X_{n+1} &= \left(I+\frac{1}{2}\tilde{\Psi}_{h, n}(\tilde{X})\right) \tilde{X} \left(I- \frac{1}{2}\tilde{\Psi}_{h, n}(\tilde{X})\right).
    \end{split}
\end{equation}
We remark that $\tilde{\Psi}_{h, n}$ is a stochastic map. Therefore, we have that the integrator $\Psi_{h}\colon \mathfrak{g}^* \times [0, T] \times \Omega \to \mathfrak{g}^*$ and that $X_{n+1}$ is a random variable.
In what follows, we will refer to this integrator as the map $\Psi_h$ defined by $X_{n+1}=\Psi_h(X_n)$ as per \eqref{eq:stochastic_midpoint_lp}. The explicit dependence on $t\in [0, T]$ and $\omega\in\Omega$ is omitted for the sake of presentation.
The first main result of the paper is to prove that the integrator \eqref{eq:stochastic_midpoint_lp} indeed is a Poisson integrator and a Lie--Poisson integrator. 

\begin{theorem}
    \label{th:poissonint}
    The map $X_{n+1}=\Psi_h(X_n)$ is a Poisson integrator on $\mathfrak{g}^*$, i.e., it preserves the Poisson bracket almost surely.
\end{theorem}
\begin{proof}
    The proof follows the arguments given by \cite{modin2020lie,modin2023spatio} and \cite[Theorem 5.11]{hairer2006geometric}. The symplectic and equivariant integrator $\Phi_h$ \eqref{eq:stoch_midpoint_hamiltonian} on $T^*G$ descends to the integrator $\Psi_h$ on $\mathfrak{g}^*$ after applying the momentum map. 
    The stochastic map $\Phi_h$ maps into $T^*G$ almost surely, after which applying the momentum map (as a mapping from $T^*G$ to $\mathfrak{g}^*$) is justified. The momentum map is a Poisson map \cite{marsden2013introduction} and hence the integrator $\Psi_h$ is a Poisson integrator. The momentum map is smooth and therefore its composition with $\Phi_h$ almost surely preserves the Poisson bracket. This result is valid for each fixed time and $\omega$.
\end{proof}

\begin{theorem}
\label{th:lpint}
    The map $X_{n+1}=\Psi_h(X_n)$ is a Lie--Poisson integrator on $\mathfrak{g}^*$, i.e., it almost surely preserves Casimirs and coadjoint orbits and is almost surely a symplectic map on the orbits.
\end{theorem}
\begin{proof}
   The result follows along the lines presented by \cite{modin2020lie,modin2023spatio}. 
   A schematic representation of the construction of the Lie--Poisson integrator is provided in \Cref{fig:illustration}, due to \cite{modin2023spatio}, and is elaborated below. 
    As a first step, recall that $\mathfrak{g}^* \cong T^*G /G$. 
    This means that points in $\mathfrak{g}^*$ can be identified with $G$-orbits in the cotangent bundle. 
    Thus, if two points in $T^*G $ belong to the same orbit, they map to the same point in $\mathfrak{g}^*$. 
    Recall that
    \begin{align*}
        \Psi_h(X_n)= \mu(\Phi_h(Q_n,P_n)). 
    \end{align*}
    Thus, if $\Psi_h$ is to be well-defined, $\Phi_h$ must map orbits to orbits, i.e.,  for any $g \in G$, 
    $\Phi_h(Q, P)$ must be in the same orbit as $\Phi_h(g\cdot(Q, P))$. This holds if 
    $\Phi_h(g\cdot(Q, P)) = g\cdot\Phi_h(Q, P)$, which is true since $\Phi_h$ is equivariant as per \Cref{lem:equivariance}.

    Now, let us verify that the mapping preserves the coadjoint orbit $\cO_{X} = \{g^* X (g^*)^{-1}|g \in G\}$ almost surely. 
    This amounts to showing that there is a $g\in G$ such that $\Psi_h(X_n) = g^*X_n (g^*)^{-1}$ almost surely. 
    The map $\Psi_h$ contains a noise increment and $\Psi_h(X_n)$ is a random variable. Therefore, the group element $g$ will depend on the noise increment and on $X_n$ and will also be a random variable.
    The steps presented here follow closely from the previous work \cite{viviani2020minimal}. 
    First note that, from the definition of the implicit midpoint scheme \eqref{eq:stochastic_midpoint_lp}, we may write \begin{equation}
        \tilde{X} = \left(I- \frac{1}{2}\tilde{\Psi}_{h, n}(\tilde{X})\right)^{-1} X_n \left(I+\frac{1}{2}\tilde{\Psi}_{h, n}(\tilde{X})\right)^{-1}
    \end{equation}
    and therefore \begin{equation}
        X_{n+1} = \left(I- \frac{1}{2}\tilde{\Psi}_{h, n}(\tilde{X})\right)^{-1}\left(I+\frac{1}{2}\tilde{\Psi}_{h, n}(\tilde{X})\right) X_n  \left(I- \frac{1}{2}\tilde{\Psi}_{h, n}(\tilde{X})\right) \left(I+\frac{1}{2}\tilde{\Psi}_{h, n}(\tilde{X})\right)^{-1},
    \end{equation}
    where we have used that $(I+A)^{-1}(I-A) = (I-A)(I+A)^{-1}$ for general $A$.
    We recall the Cayley transform, which maps elements of $\mathfrak{g}$ into $G$ and is a local diffeomorphism near $A=0$.
    For any $A$ in a $J$-quadratic Lie algebra $\mathfrak{g}$ \cite{hairer2006geometric}, the resulting group element and its inverse are defined as \begin{equation}
    \begin{split}
        \mathrm{cay}(A) &= (I - A)^{-1} (I+A), \\
        \mathrm{cay}^{-1}(A) &= (I+A)^{-1} (I-A).
    \end{split}
    \end{equation}
    We thus observe that \begin{equation}
        X_{n+1} = \mathrm{cay}^{-1}\left(-\frac{1}{2} \tilde{\Psi}_{h, n}(\tilde{X}) \right) ~ X_n ~ \mathrm{cay}\left(-\frac{1}{2} \tilde{\Psi}_{h, n}(\tilde{X}) \right).
    \end{equation}
    The operator $\tilde{\Psi}_{h, n}$ maps from $\mathfrak{g}^*\times [0, T] \times \Omega $ to $\mathfrak{g}^*$ since it is a linear combination of the adjoints of gradients of the Hamiltonians, each of which are elements of $\mathfrak{g}^*$. We may identify $\tilde{\Psi}_{h, n}(\tilde{X})\in \mathfrak{g}^*$ with an element of the Lie algebra $\mathfrak{g}$ through the dual pairing. The Cayley transform maps into the compact group $G$ and is therefore surjective \cite{hall2013lie}.
    We can thus define the $G$-valued random variable~$g$ such that $g^*=\mathrm{cay}^{-1}\left(-\frac{1}{2} \tilde{\Psi}_{h, n}(\tilde{X}) \right)$, which proves that the mapping preserves coadjoint orbits almost surely. 

    Finally, we note that the group action preserves the symplectic structure. As shown in the previous part of the proof, the outcome of the mapping $\Psi_h(X_n)$ can be identified with a group element. This group element can subsequently be written as a combination of left and right group actions acting on $X_n$. In addition, the integrator is a Poisson map by \Cref{th:poissonint}. As an immediate consequence of these results, the integrator almost surely preserves the symplectic structure on the coadjoint orbits.
\end{proof}

\begin{figure}[h!]
\centering
\begin{tikzpicture}[scale=0.9]
\draw[thick] (0,0) rectangle (10,6);
\draw[thick] (3,0.3) -- (3,5.7);
\draw[thick] (6,0.3) -- (6,5.7);
\draw[thick] (2,-4) -- (7.5,-4);
\filldraw[black] (3,2) circle (0.1);
\filldraw[black] (3,4) circle (0.1);
\filldraw[black] (6,2) circle (0.1);
\filldraw[black] (6,4) circle (0.1);
\filldraw[black] (3,-4) circle (0.1);
\filldraw[black] (6,-4) circle (0.1);
\draw [->, thick,dashed,-{Stealth[bend,length=10pt]}] (3,-0.5) -- (3,-3.5);
\draw [->, thick,dashed,-{Stealth[bend,length=10pt]}] (6,-0.5) -- (6,-3.5);
\draw [->, thick,-{Stealth[bend,length=10pt]}] (3,-4.5) .. controls  +(down:10mm) and +(down:10mm) .. (6,-4.5);
\draw[thick] (0,-6) -- (7,-6);
\draw[thick] (0,-6) -- (2,-2.5);
\draw[thick] (2,-2.5) -- (9,-2.5);
\draw[thick] (7,-6) -- (9,-2.5);
\node[scale=1.5] at (4.5,-5.5) {$\Psi_h$};
\node[scale=1.5] at (2.5,2) {$a$};
\node[scale=1.5] at (7,1.5) {$\Phi_{h}(a)$};
\node[scale=1.5] at (2.3,4) {$g \cdot a$};
\node[scale=1.5] at (7.5,4.7) {$(\Phi_{h})(g\cdot a)$};
\draw [->, thick,-{Stealth[bend,length=10pt]}] (6.2,2) .. controls  +(right:5mm) and +(right:5mm) .. (6.2,4);
\node[scale=1.5] at (7,3) {$g$};
\draw [->, thick,-{Stealth[bend,length=10pt]}] (3.2,2) .. controls  +(up:5mm) and +(up:5mm) .. (5.8,2);
\draw [->, thick,-{Stealth[bend,length=10pt]}] (3.2,4) .. controls  +(up:5mm) and +(up:5mm) .. (5.8,4);
\node[scale=1.5] at (4.5,4.8) {$\Phi_{h}$};
\node[scale=1.5] at (2,-4.5) {$\mathcal{O}$};
\node[scale=1.5] at (4.5,1.8) {$\Phi_{h}$};
\node[scale=1.5] at (2.5,-1.5) {$\mu$};
\node[scale=1.5] at (5.5,-1.5) {$\mu$};
\node[scale=1.5] at (11.5,-3.8) {$\mathfrak{g}^{*}=T^{*}G/G$};
\node[scale=2.5] at (11.5,4) {$T^{*}G$};
\draw [->, thick,-{Stealth[bend,length=10pt]}] (11.5,3.5) -- (11.5,-2.5);
\node[scale=2] at (12,0.5) {$\mu$};
\end{tikzpicture}
\caption{Schematic representation of the construction of the Lie--Poisson integrator, as in \cite{modin2023spatio}. We consider an equivariant symplectic method $\Phi_h\colon T^*G\to T^*G$, i.e., satisfying $\Phi_h(g\cdot a) = g\cdot\Phi_h(a)$ for $g\in G$ and $a\in T^*G$.
This method descends to a Lie--Poisson method $\Psi_h\colon\mathfrak{g}^*\to\mathfrak{g}^*$ on the coadjoint orbit $\mathcal{O}$ after applying the momentum map $\mu$.}
\label{fig:illustration}
\end{figure}
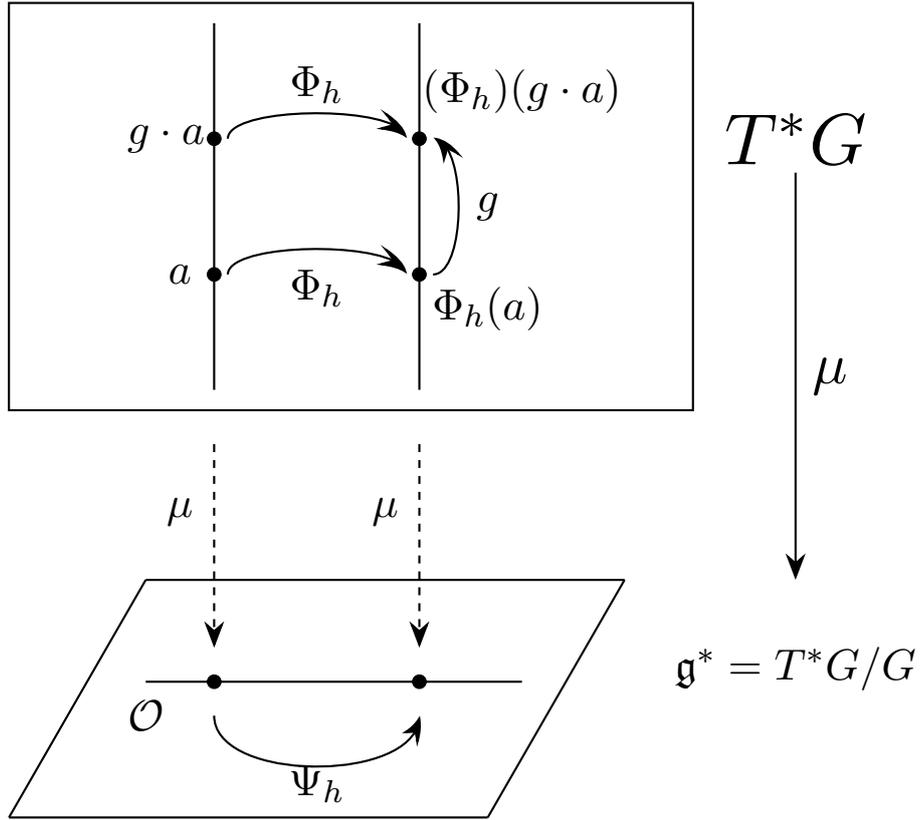

\section{Error analysis}
\label{sec:error_analysis}

In this section, we bound the strong and weak error of the scheme \eqref{eq:stochastic_midpoint_lp}. 
As it is derived from a well-studied integrator for Hamiltonian systems, the midpoint method, we can bound the convergence of the scheme \eqref{eq:stochastic_midpoint_lp} by the convergence properties of the midpoint method applied to the canonical Hamiltonian system. 
To this end, the first step is to verify the strong and weak convergence for the midpoint method \eqref{eq:stoch_midpoint_hamiltonian}. 
Strong and weak convergence, under several strong assumptions, have been established for stochastic Hamiltonian systems in, e.g.,~\cite{Milstein2004}. 
Here, we rely on the geometric structure of the stochastic Hamiltonian system \eqref{eq:hamsystem} to prove convergence and thereby avoid overly restrictive assumptions on the coefficients.
Therefore, we use the same type of truncation argument as in the proof of \Cref{lem:upstairs_existence} to prove convergence of the implicit midpoint method applied to the system \eqref{eq:hamsystem}. 

As a first step to proving convergence, we have the following technical lemma, adapted from the Lie--Poisson case of \cite[Proposition 3.3]{brehier2023splitting}. 
\begin{lemma}
    \label{lem:numbound}
    Consider the stochastic Hamiltonian integrator $(Q_{n+1},P_{n+1}) = \Phi_h(Q_n,P_n)$ defined by \eqref{eq:stoch_midpoint_hamiltonian}.
    Then, it holds that for any initial value $(Q_0,P_0) \in T^*G$ that there exists an $R(Q_0,P_0)<\infty$ such that almost surely,
    \begin{align}
        \label{eq:numer_bound}
        \sup_{h \geq 0} \sup_{n \geq 0} \|Q_n,P_n\| \leq R(Q_0,P_0). 
    \end{align}
 \end{lemma}
\begin{proof}
    By \Cref{th:lpint}, the scheme \eqref{eq:stoch_midpoint_hamiltonian} descends to a Casimir preserving integrator, 
    and for each step $\mu \circ \Phi_h(Q_n,P_n) = \mu((Q_{n+1},P_{n+1}))$ corresponds to a step with the Casimir preserving integrator on the dual of the algebra. 
    Thus, for all $n \geq 1$, for an arbitrary Casimir $C$  with compact level sets
    \begin{align*}
        C(\mu(Q_{n},P_{n})) = C(\mu(Q_{n-1},P_{n-1})) = \dots = C(\mu(Q_0,P_0)).
    \end{align*}
    Now, by Lemma \ref{lem:upstairs_existence}, since $R(Q_0,P_0)$ is the maximum of $\|(Q,P)\|$ over the level set of $(Q_0,P_0)$ of the Casimir, we have by the compactness of the level set that 
    \begin{equation*}
        \|Q_n,P_n\| \leq R(Q_0,P_0). 
        \qedhere
    \end{equation*}
\end{proof}

We now establish convergence of the implicit midpoint method for system \eqref{eq:hamsystem}.
We remark that the Casimir preservation proven in \Cref{th:lpint} is central to proving convergence. In fact, there is no guarantee of convergence without Casimir preservation. As a first step, we prove the convergence of the truncated stochastic Hamiltonian system \eqref{eq:hamsystem_trunc}. 

\begin{lemma}
    \label{lem:upstairs_convergence}
    Let $H_0$ be of class~$C^4$ and let $H_1,H_2, \ldots, H_M$ be of class~$C^5$.
    Consider the stochastic Hamiltonian integrator $(Q_{n+1}^R,P_{n+1}^R) = \Phi_h(Q_n^R,P_n^R)$ applied to the truncated stochastic Hamiltonian system \eqref{eq:hamsystem_trunc}. 
    Let $T >0$ be an arbitrary bounded final time and let $N$ be the number of steps needed to reach the final time with constant step size $h=T/N$. 
    Then, there is a constant~$\kappa>0$ such that 
    \begin{align*}
        \sup_{0 \leq n \leq N} \E[\left\|\left(Q^R(hn),P^R(hn)\right)-(Q_n^R,P_n^R)\right\|^2]^{1/2} \leq \kappa h^{1/2}. 
    \end{align*}
    Further, if $C^4 \ni \phi: T^*G \to \R$ has bounded derivatives, then there is a constant $\kappa>0$ such that 
    \begin{align}
    \sup_{0 \leq n \leq N} \left|\E[\phi(Q^R(hn),P^R(hn))-\phi(Q_n^R,P_n^R)]\right| \leq \kappa h. 
    \end{align}
\end{lemma}
\begin{proof}
    Consider a general Stratonovich SDE, 
    \begin{align*}
        dY_t = \mu(Y_t)\mathrm d t + \sum_{i=1}^m \sigma_i(Y_t)\circ \mathrm d W_t^i.
    \end{align*}
    The midpoint method is known to convergence strongly and weakly for systems of this form under strong assumptions on  $\mu,\sigma_1, \dots, \sigma_m$. 
    All coefficients, as well as the Stratonovich corrections 
    and $\partial_x \sigma_i \sigma_i$ must be Lipschitz. 
    The drift coefficient must be in $C^4$ with bounded derivatives, and 
    the diffusion coefficients $\sigma_1,\ldots, \sigma_m$ must be in $C^5$ with bounded first and second derivatives. 
    To prove convergence in our case, we first verify that the truncated canonical system \eqref{eq:hamsystem_trunc} satisfies these conditions \cite[Section 2.2]{milstein2002numerical}, \cite[Section 2.2.1,Theorem 6.1]{Milstein2004}. 
    
    Indeed, the coefficients of \eqref{eq:hamsystem_trunc} are all at least continuously differentiable and have compact support, and so, the coefficients are globally Lipschitz. 
    Moreover, the compact support in combination with the continuity assumptions ensures the requisite boundedness. 
    Therefore, by \cite{Milstein2004, hong2023symplectic}, strong and weak convergence of the implicit midpoint method applied to  system~\eqref{eq:hamsystem_trunc} follow. 
    In detail, there exists a constant~$\kappa$ such that 
       \begin{align*}
        \sup_{0 \leq n \leq N} \E[\|(Q^R(hn),P^R(hn))-(Q_n^R,P^R_n)^2\|]^{1/2} \leq \kappa h^{1/2}, 
    \end{align*}
    where $(Q_n^R,P_n^R)$ denotes the $n$-th step of the integrator applied to the truncated system \eqref{eq:hamsystem_trunc}.
    Further, \cite{Milstein2004, hong2023symplectic}, ensures that for any test function $C^{4} \ni \phi:T^*G \to \R  $ with the property that it and its derivative have at most polynomial growth, it holds that there is a constant~$\kappa>0$ (possibly depending on $\phi$) such that 
    \begin{align*}
        \sup_{0 \leq n \leq N} \left|\E[\phi(Q^R(hn),P^R(hn))-\phi(Q_n^R,P^R_n)]\right| \leq \kappa h,
    \end{align*}
    meaning that the implicit midpoint method applied to the truncated system \eqref{eq:hamsystem_trunc} converges with root mean squared order $1/2$ and weak order $1$. 
   \end{proof}
\begin{remark}
    It is possible to prove convergence of the midpoint method applied to the full Hamiltonian system. 
     Indeed, one can apply \Cref{lem:numbound} to prove convergence for the implicit midpoint method. 
    Since, as in the proof of \Cref{lem:upstairs_existence}, it holds by the Casimir preservation
    that $(Q^R_n,P^R_n) = (Q_n,P_n)$ and $(Q^R,P^R) = (Q,R)$. 
    This is, however, not necessary to prove the convergence of the stochastic isospectral midpoint integrator. 
\end{remark}
We are now ready to prove the convergence of the isospectral midpoint method \eqref{eq:stochastic_midpoint_lp}. 
\begin{theorem}
    \label{th:convergence}
    Let $H_0$ be of class $C^4$ and let $H_1,H_2, \ldots, H_M$ be of class $C^5$.
        Consider the stochastic Lie--Poisson integrator $X_{n+1} = \Psi_h(X_n) = \mu(\Phi_h(Q_n,P_n))$  applied to the system~\eqref{eq:sys_strat}. 
      Let $T >0$ be an arbitrary bounded final time and let $N$ be the  number of steps needed to reach the final time with constant step size $h = T/N$.
      Then, for any deterministic initial condition $X_0 \in \mathfrak{g}^*$, there is a constant $\kappa >0$ such that 
        \begin{align*}
             \sup_{0 \leq n \leq N} \E\left[ \|X(hn)-X_n\|^2\right]^{1/2} \leq \kappa h^{1/2},
        \end{align*}
        i.e., the root mean squared order of convergence is $1/2$. 
        Further, for any  $C^4 \ni \phi: \mathfrak{g}^* \to \R$ with bounded derivatives, it holds that 
        \begin{align*}
             \sup_{0 \leq n \leq N} \left|\E[\phi(X(hn))-\phi(X_n)]\right| \leq \kappa h,
        \end{align*}
        i.e., the weak order of convergence is $1$. 
\end{theorem}
\begin{proof}
    Consider the solution $X^R_t$ of the truncated stochastic Lie--Poisson system \eqref{eq:sys_strat_trunc}. 
    Note that  $X^R_t =\mu(Q^R_t,P^R_t)$ and that 
    $X^R_n = \mu(Q^R_n,P^R_n)$, where $(Q^R_n,P^R_n)$ is the $n$-th step of the integrator \eqref{eq:stoch_midpoint_hamiltonian} applied to the system \eqref{eq:hamsystem_trunc}. 
    For the mean square convergence of the truncated system, note that $\mu$ restricted to the compact tube $\mu^{-1}(B_R)$ is a smooth function mapping from a compact set into another compact set, meaning in particular that it is a Lipschitz map.
    Therefore, for any $n$, there is a constant $\kappa>0$ such that 
    \begin{align*}
        \|X^R(hn)-X^R_n\| = \|\mu(Q^R(hn),P^R(hn))-\mu(Q^R_n,P^R_n)\| \leq \kappa\|(Q^R(hn),P^R(hn))-(Q^R_n,P^R_n)\|,
    \end{align*}
    after which we apply 
   \Cref{lem:upstairs_convergence} to see that there is a constant $\kappa>0$ such that 
    \begin{align*}
        \sup_{0 \leq n \leq N} \E\left[ \|X^R(hn)-X_n^R\|^2\right]^{1/2}  \leq \sup_{0 \leq n \leq N} \E\left[ \|(Q^R(hn),P^R(hn))-(Q^R_n,P^R_n)\|^2\right]^{1/2} \leq \kappa h^{1/2}.  
    \end{align*}

    To prove weak convergence for the truncated system, note that $\mu$ restricted to the compact tube $\mu^{-1}(B_R)$ is a smooth function mapping from a compact set into another compact set, and therefore it and its derivatives are bounded.
    Thus, $\phi \circ \mu: T^*G \to \R$ is a $C^4$ function with bounded derivatives, and we apply \Cref{lem:upstairs_convergence} to obtain that there is a constant $\kappa>0$ such that 
    \begin{align*}
        \sup_{0 \leq n \leq N} \left|\E[\phi(X^R(hn))-\phi(X^R_n)]\right|
        = \sup_{0 \leq n \leq N} \left|\E[\phi\circ \mu(Q^R(hn),P^R(hn))-\phi\circ \mu(Q_n^R,P^R_n)]\right|
        \leq \kappa h. 
    \end{align*}
    Having thus established the root-mean-square and weak convergence of the method \eqref{eq:stochastic_midpoint_lp} applied to the truncated system, the theorem follows by noting that just as in the proof of \Cref{lem:upstairs_convergence}, $X^R_t = X_t$.
    Further, by \Cref{lem:numbound}, $X^R_n$ remains in the ball~$B_R$, meaning that $ T_i(X^R_n) = H_i(X^R_n)$, for all $i = 0, \ldots, M$.
    This implies in particular that  $X^R_n = X_n$, which completes the proof. 
\end{proof}

\section{Examples and numerical experiments}
\label{sec:examples}
In this section, we present various examples of stochastic Lie--Poisson systems and verify the theoretical results of the stochastic integrator with several numerical experiments. We consider the three-dimensional rigid body motion and its $n$-dimensional generalization, point vortex dynamics on the sphere, and the two-dimensional Euler equations for inviscid fluid flow on the sphere. For each example, we demonstrate Casimir preservation and strong convergence. Weak convergence is only shown numerically for the three-dimensional rigid body. This is due to the large number of simulations needed, meaning that the computational resources required in higher-dimensional systems is prohibitively large.

The noise generated for all realizations is truncated to implement the stochastic midpoint scheme. Throughout the numerical experiments, the truncation threshold is chosen as $A_h = \sqrt{4|\ln(h)|}$ \cite{milstein2002numerical}. Furthermore, the numerical realizations obtained with time step size $h=2^{-16}$ are used as reference results in the convergence tests for each system. These results are compared to the realizations obtained at a range of coarser time steps $h=2^{-7}, \ldots, 2^{-12}$ for strong convergence and $h=2^{-7}, \ldots, h=2^{-13}$ for weak convergence.

The implicit step of the stochastic integrator is solved via fixed point iteration. In all performed numerical experiments, the tolerance of the iterative process is set to $10^{-15}$ and is measured in the infinity norm.

    \subsection{The rigid body equations}
    The rigid body is a canonical example among mechanical systems with conservation laws. It describes the equations of motion of a three-dimensional rigid body rotating about a fixed point without external forcing or damping. The deterministic Hamiltonian is given by \begin{equation}
        H_0(X) = \frac{1}{2}\trace((\mathcal{I}^{-1}X)^*X),
        \label{eq:ham_rigid_body}
    \end{equation}
    where $X \in \mathfrak{so}(3)$ and $\mathcal{I}:\mathfrak{so}(3)\to\mathfrak{so}(3)$ denotes the moment of inertia tensor. The stochastic Lie--Poisson equations read \begin{equation}
        \begin{split}
            \dd X &= -\left[\mathcal{I}^{-1}X, X \right]\, \dd t + \sum_{i=1}^M\left[\nabla H_i(X)^*, X \right]\circ\dd W_t^i, \\
            X(0) &= X_0.
        \end{split}
        \label{eq:rigid_body}
    \end{equation}
    The hat-map yields an isomorphism between $\R^3$ and $\mathfrak{so}(3)$. For $x = (x_1, x_2, x_3)^T\in\R^3$, the hat-map is defined as
    \begin{equation}
        \hat{x} = \begin{bmatrix} 0 &-x_3& x_2 \\ x_3& 0& -x_1 \\ -x_2 & x_1 & 0
        \end{bmatrix}.
    \end{equation}
    All realizations of the rigid body presented here make use of the inertia tensor defined through the hat map and its inverse. The inertia tensor may be represented as a diagonal $3\times 3$ matrix, denoted by $I=\Diag(I_1, I_2, I_3)$. 
    Its action on an element $\hat{x}\in\mathfrak{so}(3)$ then takes the form of $\widehat{Ix}$, where ordinary matrix-vector multiplication between $I$ and $x$ applies. 
    
    The adopted initial condition is $X_0=\hat{x}_0$, where $x_0 = (\sin(1.1), 0, \cos(1.1))^T$ and an inertia tensor given by $\Diag(2, 1, 2/3)$. 
    Three noise Hamiltonians are added in the stochastic system. 
    For $i=1,2,3$ these noise Hamiltonians $H_i$ are chosen such that $\nabla H_i(X) = 0.1\widehat{x_i \mathbf{e}_i}$, where $\mathbf{e}_i$ is the standard $i$-th basis vector of $\R^3$.

    Conservation of the eigenvalues in the stochastic system is demonstrated using a long-time simulation. A time step size of $h=2^{-8}$ time unit is adopted to simulate a total of 1000 time units, or approximately $2.5\times 10^5$ time steps. 
    As shown in  \Cref{fig:rigbod_pres}, the eigenvalues are preserved up to machine precision during the simulation and the Hamiltonian changes over time. 
    Strong convergence of the integrator is illustrated in the left panel of \Cref{fig:rigbod_strong}.
    A total of 500 realizations are carried out for each time step size, simulating for $0.1$ time units. We observe a strong order of convergence of $1/2$, in agreement with \Cref{th:convergence}. 

    A weak convergence test has been carried out by simulating $10^7$ realizations for each time step size. The solutions are compared after $0.1$ time units with the test function proposed by \cite{brehier2023splitting}, namely $\phi(X)=\sin (2\pi x_1) + \sin(2\pi x_2) + \sin(2\pi x_3)$. The results are depicted in the right panel of \Cref{fig:rigbod_strong} and suggest first-order weak convergence. 
    Here, the number of Monte Carlo samples in the demonstration of the weak convergence needs to be chosen much larger than in the strong convergence, since the Monte Carlo error is additive and easily dominates the other error contributions, while it just contributes multiplicatively to the error constant in the strong error simulations, see \cite{lang2018monte}.

        \begin{figure}
        \centering
        \includegraphics[width=1\textwidth]{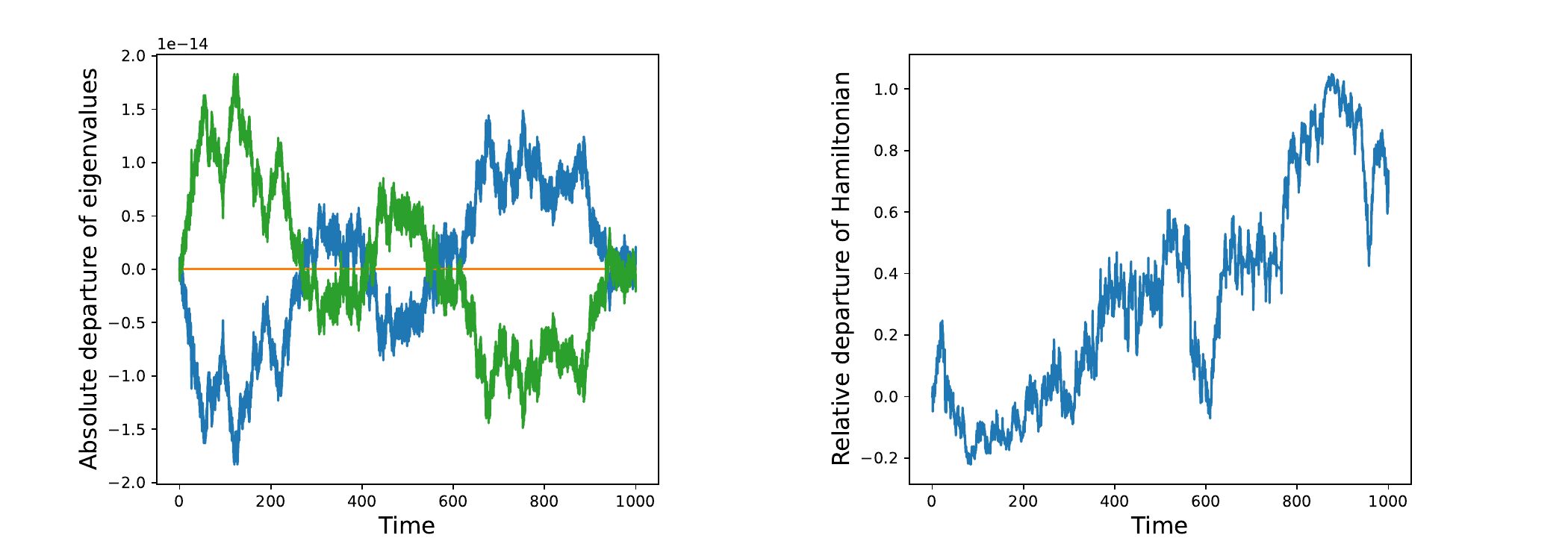}
        \caption{A single realization of the stochastic rigid body. Left: absolute departure of the eigenvalues of the solution matrix $X$, measured from the initial eigenvalues. Right: relative departure of the Hamiltonian, normalized by the initial value. 
        }
        \label{fig:rigbod_pres}
    \end{figure}

    \begin{figure}
        \centering
        \includegraphics[width=0.48\textwidth]{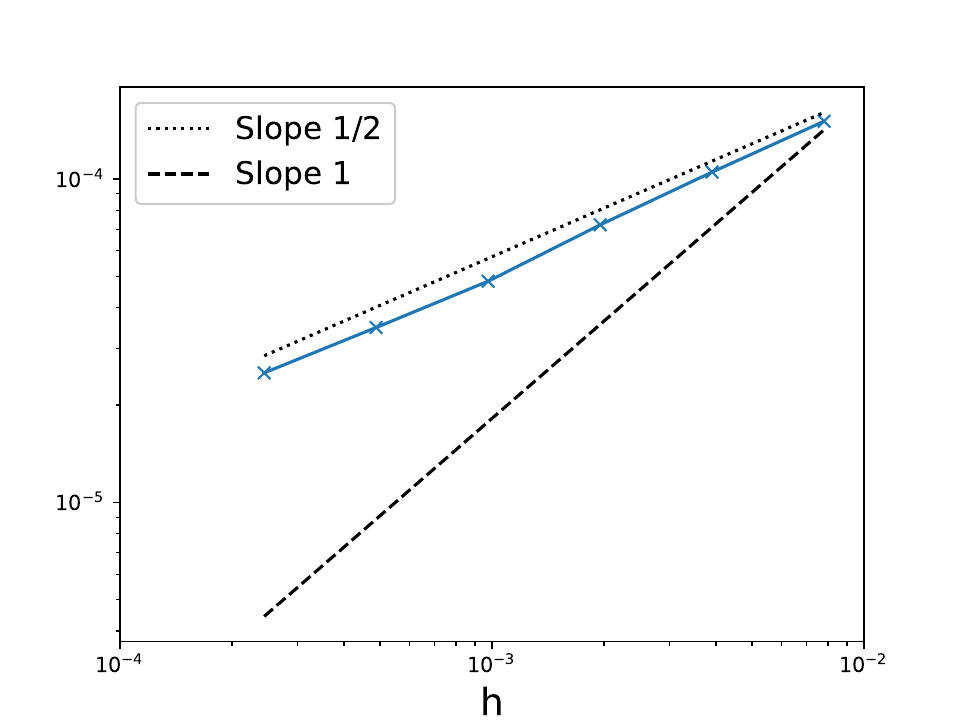}
        \includegraphics[width=0.48\textwidth]{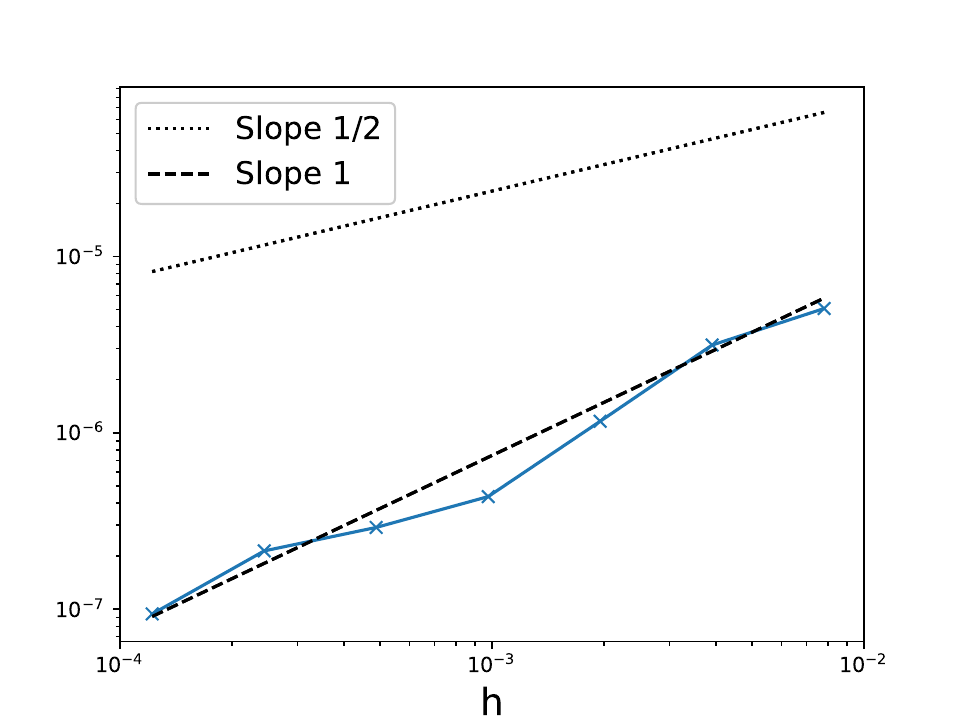}
    \caption{Left: Strong errors for the rigid body system using 500 independent realizations. Right: Weak error for the rigid body system based on $10^7$ independent realizations. 
    }
    \label{fig:rigbod_strong}
\end{figure}
    
    \subsection{Manakov system}
    The Manakov system describes the motion of an $n$-dimensional rigid body rotating about a fixed point \cite{ratiu1980motion, manakov1976note}. This system is described by the same Hamiltonian  and Lie--Poisson equations as the three-dimensional rigid body, the only changes being that the dual of the Lie algebra is identified with $\mathfrak{so}(n)$. As a result, $X\in\mathfrak{so}(n)$ and the moment of inertia tensor is defined as a map $\mathcal{I}: \mathfrak{so}(n)\to\mathfrak{so}(n)$.

    The inertia tensor is chosen here such that for the upper triangular part of $X$ \begin{equation}
        (\mathcal{I}X)_{ij} = \left(\sum_{k=1}^{i-1}(n-k)+(j-i)\right)X_{ij}, \quad i=1,\ldots, n;\quad  j=i+1, \ldots, n,
    \end{equation}
    and the lower triangular part follows from skew-symmetry of $X$. This system admits a total of $n(n-1)/2$ degrees of freedom. In the numerical experiments, $n(n-1)/2$ noise Hamiltonians are selected such that \begin{equation}
        \nabla H_{ij}(X) = \begin{cases} 0.1 X_{kl} &\text{ if } (k, l) \in \{ (i, j), (j, i) \} \\
        0 &\text{ elsewhere,}
        \end{cases}
    \end{equation}
    where $i=1,\ldots, n$ and $j=i+1,\ldots,n$.
    
    The conservation of eigenvalues for the stochastic Manakov system is shown in \Cref{fig:manakov_pres}. Here, $n=10$ and the system is simulated for 1000 time units. The strong error is again computed by comparing 500 realizations after $0.1$ time units. As may be seen in \Cref{fig:manakov_strong}, strong convergence of order 1/2 is observed. These results illustrate that the integrator can readily be extended to higher-dimensional systems without loss of accuracy or of conservation properties.

    \begin{figure}
        \centering
        \includegraphics[width=1\textwidth]{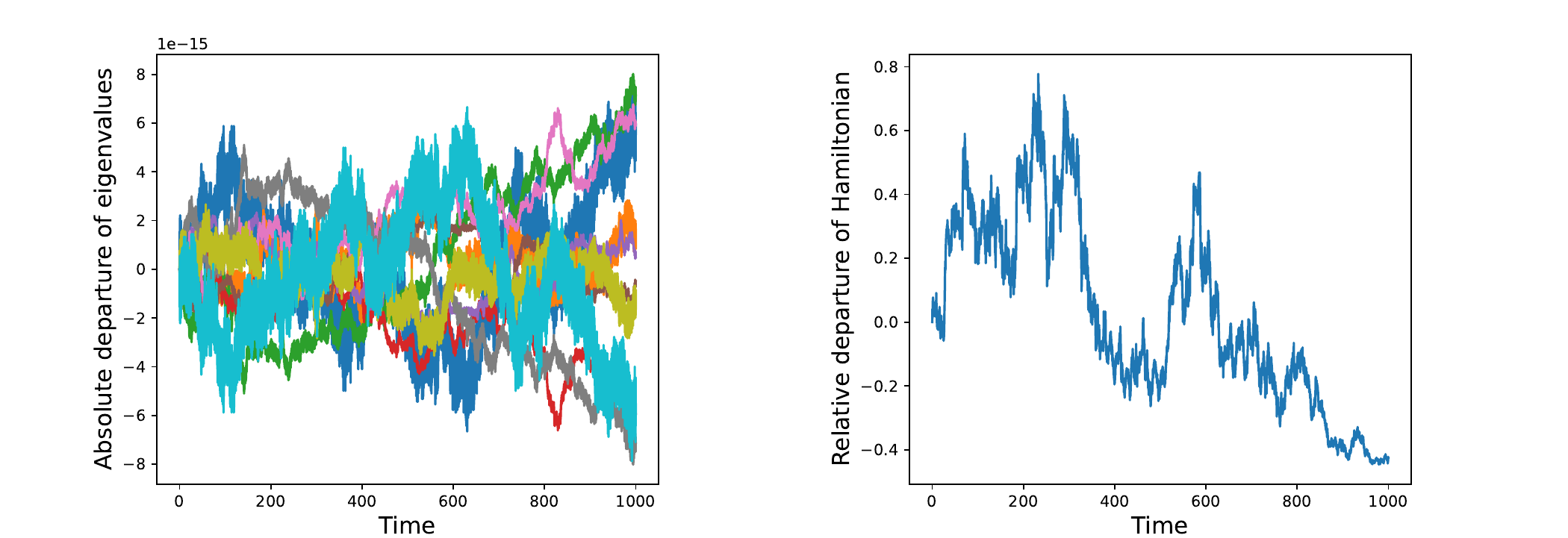}
        \caption{A single realization of the stochastic Manakov system. Left: absolute departure of the eigenvalues of the solution matrix $X$, measured from the initial eigenvalues. Right: relative departure of the Hamiltonian, normalized by the initial value.}
        \label{fig:manakov_pres}
    \end{figure}
    \begin{figure}
        \centering
        \includegraphics[width=0.5\textwidth]{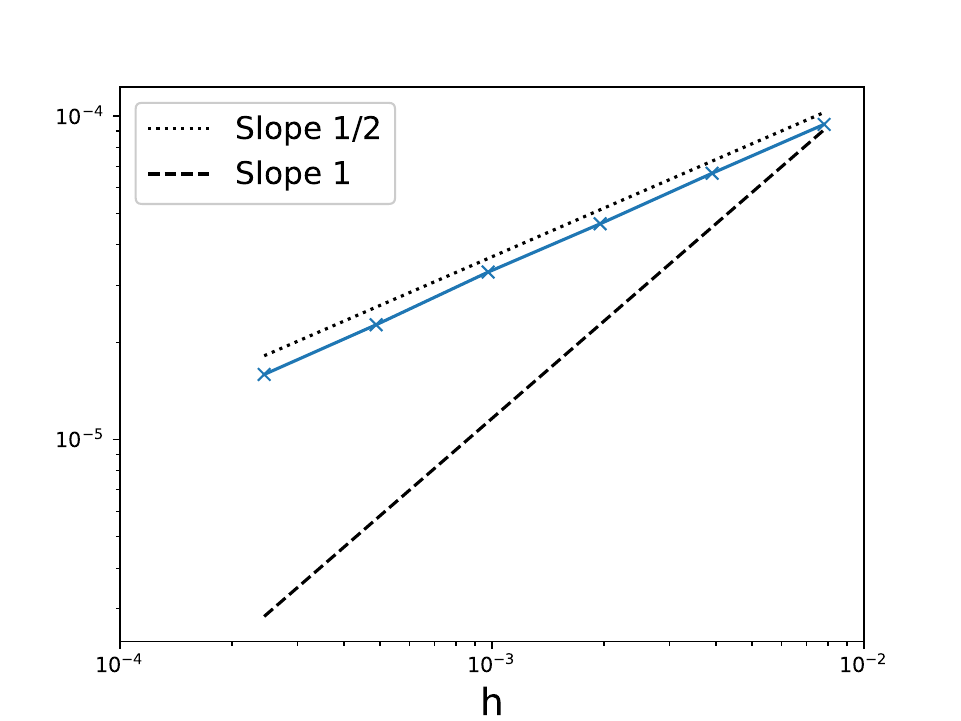}
        \caption{Strong errors for the Manakov system using 500 independent realizations.}
        \label{fig:manakov_strong}
    \end{figure}

    \subsection{Point vortex dynamics on the sphere}
    Point vortex dynamics describe the evolution of vorticity concentrated in a finite number of points. Specifically, Helmholtz demonstrated that these dynamics yield a solution to the two-dimensional Euler equations \cite{helmholtz1858integral}. Point vortex systems have been applied to study and predict the formation of coherent structures in two-dimensional fluid flows through statistical mechanics \cite{onsager1949statistical}. Other studies have used this system to investigate the trajectories of large-scale structures in two-dimensional turbulent flows using simulation \cite{modin2020casimir} and analysis \cite{modin2021integrability}. Finally, point vortices can also serve as a discretization for two-dimensional fluid mechanical systems \cite{chorin1973discretization}. 
    
    The Hamiltonian for the point vortex system on the sphere reads
    \begin{equation}
        H_0(X_1, \ldots, X_n) = -\frac{1}{4\pi} \sum_{j=1}^n \sum_{i=1}^{j-1} \Gamma_i \Gamma_j \log\left(1- \frac{\trace(X_i^* X_j)}{\|X_i \|^2 \|X_j \|^2}\right),
    \end{equation}
    where $\Gamma_i, i=1-e,\ldots, n$, denote the intensities of the corresponding point vortices. Each $X_i, i=1, \ldots, n$, is an element of $\mathfrak{su}(2)^*$ and can be thought of as a vector in $\R^3$ through an isomorphic mapping between the former two spaces. As such, each $X_i$ represents the position of the corresponding point vortex. 
    
    The dynamics of four point vortices are simulated in the reported numerical experiments. The vortex intensities $\Gamma_i$ are all chosen to be 1. We denote the position of the $j$-th coordinate in $\R^3$ of the $i$-th point vortex by $x_{ij}$ and adopt three noise Hamiltonians defined by \begin{equation}
        \nabla H_i(X_j) = 0.1 x_{ij}, \quad j=1, 2, 3, \quad i=1,\ldots, n.
    \end{equation}
    As shown in \Cref{fig:pvx_pres}, the eigenvalues of the solution matrix $X$ are preserved accurately over long simulation times. The strong errors, depicted in \Cref{fig:pvx_strong}, indicate a strong order of convergence of $1/2$.

    \begin{figure}
        \centering
        \includegraphics[width=1.0\textwidth]{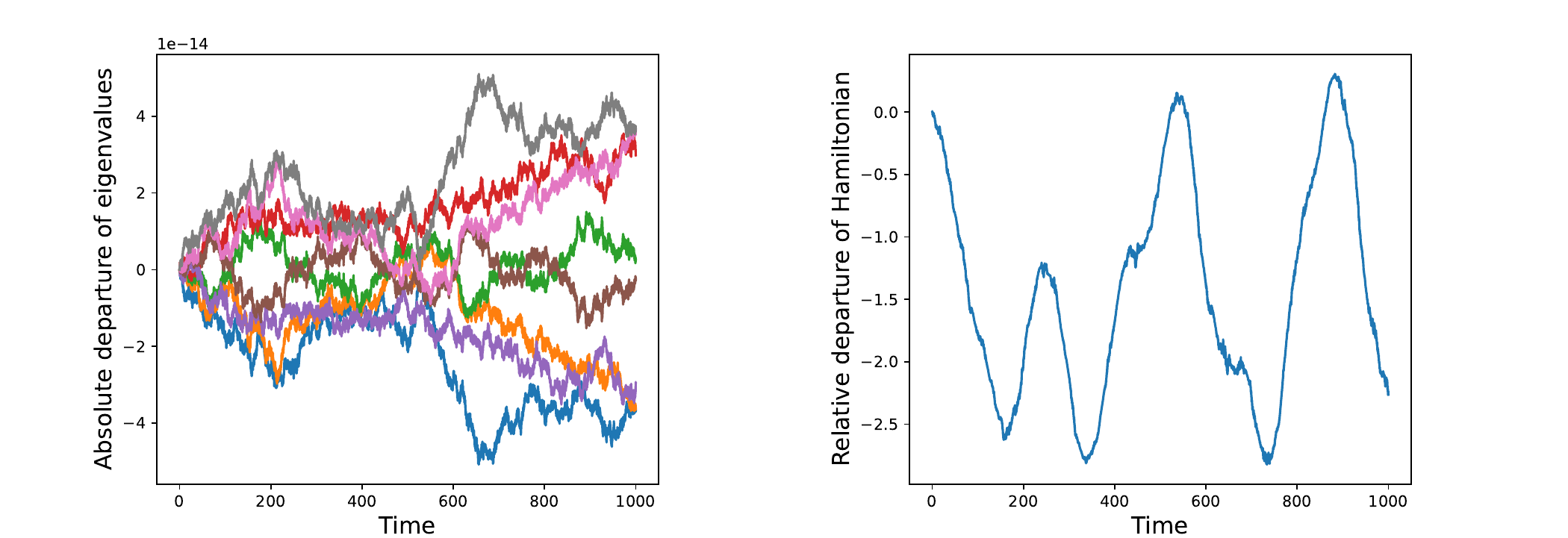}
        \caption{A single realization of the stochastic point vortex dynamics with 4 point vortices. Left: absolute departure of the eigenvalues of the solution matrix $X$, measured from the initial eigenvalues. Right: relative departure of the Hamiltonian, normalized by the initial value.}
        \label{fig:pvx_pres}
    \end{figure}
    \begin{figure}
        \centering
        \includegraphics[width=0.5\textwidth]{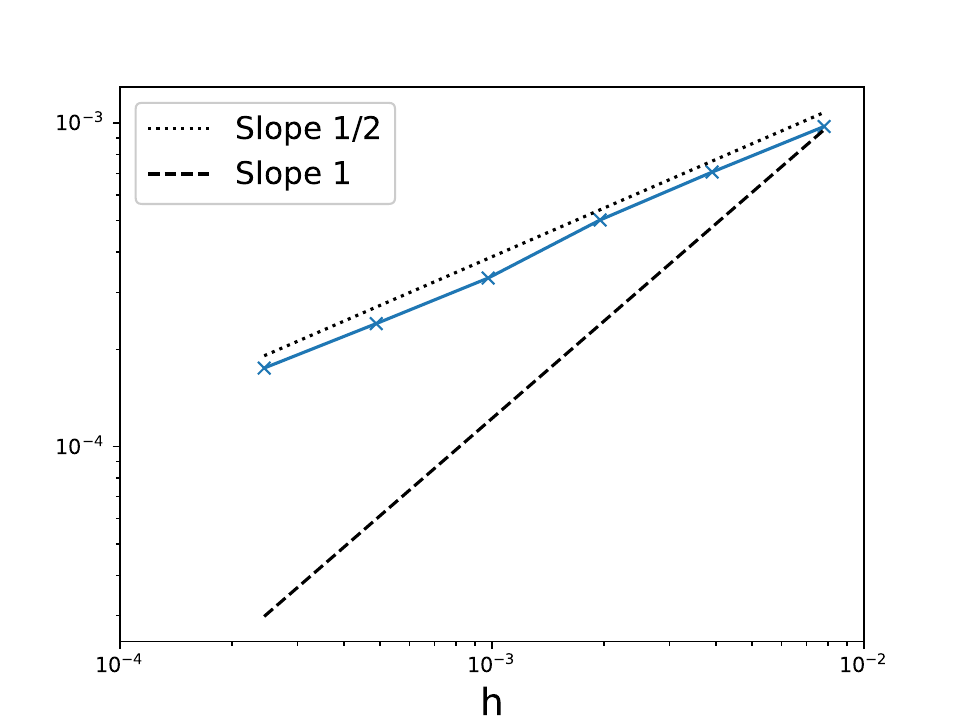}
        \caption{Strong errors for the point vortex system using 500 independent realizations.}
        \label{fig:pvx_strong}
    \end{figure}
    
    \subsection{Zeitlin--Euler equations}
    The Zeitlin--Euler equations provide a spatial discretization of the two-dimensional Euler equations through a consistent finite-dimensional truncation of the dynamics \cite{zeitlin1991finite, zeitlin2004self, modin2020casimir}. The Euler equations describe an incompressible inviscid fluid, with the evolution fully determined by the vortex dynamics \cite{zeitlin2018geophysical}, and are part of a larger class of geophysical fluid-dynamical models~\cite{holm2021stochastic}. This particular model has recently been adopted as discretization method of the convective term in the two-dimensional Navier--Stokes equations~\cite{cifani2023efficient} to provide numerical evidence of a double scaling law in two-dimensional turbulence \cite{cifani2022casimir}.

    A finite-dimensional isospectral Hamiltonian approximation of the two-dimensional Euler equations is obtained via geometric quantization \cite{hoppe1989diffeomorphism, bordemann1994toeplitz, bordemann1991gl}. In this approach, smooth functions are approximated by skew-Hermitian matrices and the Poisson bracket of functions is replaced by the matrix commutator. The resulting finite-dimensional system can be regarded as a discretization of the Euler equations on the sphere and reads
    \begin{equation}
        \begin{split}
        \dot{X} &= \left[\Delta_N^{-1} X, X\right], \\
        X(0) &= X_0,
        \end{split}
    \end{equation}
    where $\Delta_N$ is the discrete Laplacian \cite{hoppe1998some}. Analogous to the continuous setting, the eigenfunctions of the discrete Laplacian provide a basis for $\mathfrak{su}(N)$. These functions are the quantized counterparts of the spherical harmonic functions and are denoted by $T_{lm}$. Here, $l=0, \ldots, N$ is the degree and $m=-l, \ldots, l$ is the order of the spherical harmonic. The variable $N$ is the spectral resolution, 
    and the system has the property that the continuous Euler equations are approximated as $N\to \infty$. The matrix $X$ is referred to as the vorticity matrix with the property $X\in\mathfrak{su}(N)^*$. The Hamiltonian of this system is given by \begin{equation}
        H_0(X) = \frac{1}{2}\trace\left((\Delta_N^{-1} X)^* X\right).
    \end{equation}
    The stochastic isospectral system thus reads \begin{equation}
        \begin{split}
            \dd X &= \left[\Delta_N^{-1} X, X \right]\, \dd t + \sum_{l=0}^N\sum_{m=-l}^l \left[\nabla H_{lm}(X)^*, X\right]\circ\dd W_t^{lm}, \\
            X(0) & = X_0.
        \end{split}
    \end{equation}
    
    A single long-time simulation was carried out at $N=128$ with step size $h \approx 1.6\times 10^{-3}$ for a total of 32000 steps. The noise Hamiltonians used in this experiment are given by 
    \begin{equation}
        \nabla H_{lm}(X) = \frac{0.1}{l(l+1)} \trace(T_{lm}^* X), \quad l=N/2, \ldots, N, \quad m=-l, \ldots, l.
    \end{equation}
    The departure of the eigenvalues of $X$ are depicted in \Cref{fig:euler_pres} along with the departure of the enstrophy and the Hamiltonian. For clarity, we have opted to display only the maximum of the measured absolute departure of the eigenvalues. It is apparent that the eigenvalues are preserved during the entire simulation interval. The enstrophy is a well-known Casimir function in geophysical fluid dynamical models and defined as the integrated square vorticity. In the Euler--Zeitlin model, the enstrophy is defined as $C_2(W):=\trace(W^2)$. This quantity can also be seen to be conserved up to high precision. Finally, the modest magnitudes of the noise Hamiltonians are found to induce a small relative departure of the total Hamiltonian.

    Strong convergence of the integrator for the stochastic Euler--Zeitlin equations is shown at a resolution $N=12$. A total of 500 realizations are compared after $0.1$ time units, using the same time step sizes as in the other numerical examples. For this test, we adopt the noise Hamiltonians defined by
    \begin{equation}
        \nabla H_{lm}(X) = \frac{2}{l(l+1)} \trace(T_{lm}^* X), \quad l=N/2, \ldots, N, \quad m=-l, \ldots, l.
    \end{equation}
    to ensure a sufficiently large contribution of the stochastic terms to the dynamics of the system. As shown in \Cref{fig:euler_strong}, the strong order of convergence is found to be $1/2$.
    
    \begin{figure}
        \centering
        \includegraphics[width=1.0\textwidth]{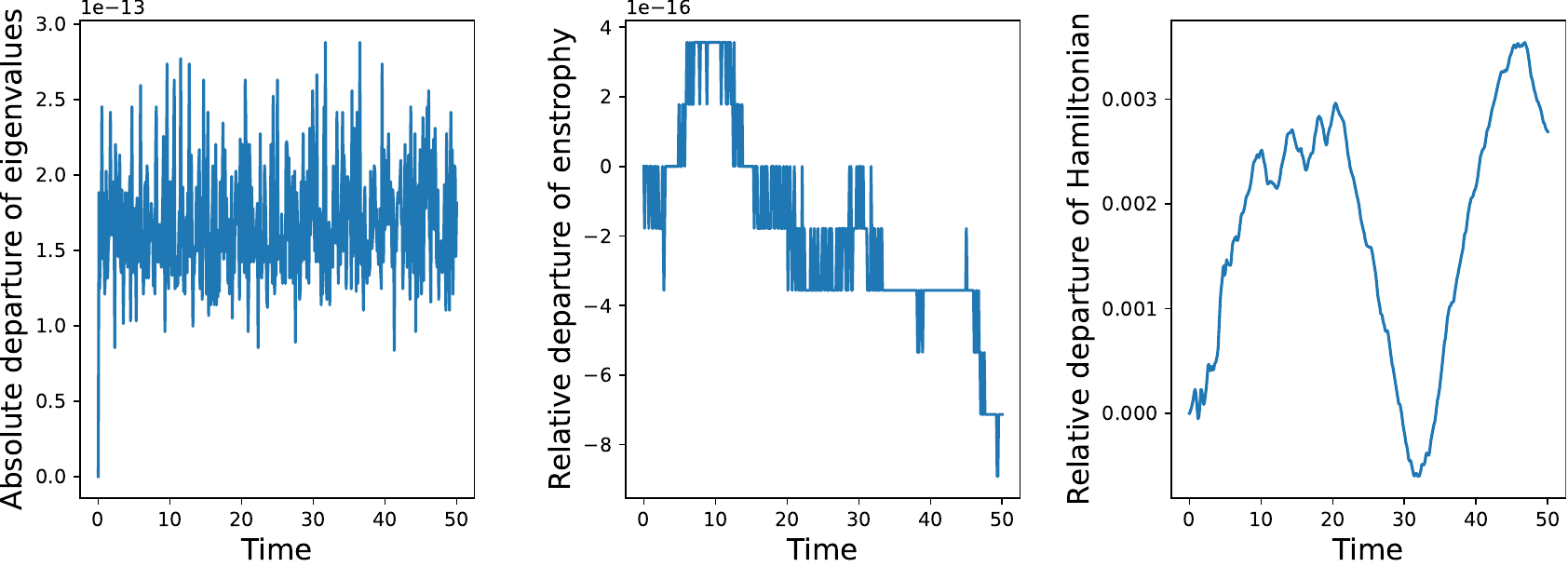}
        \caption{A single realization of the stochastic Euler--Zeitlin system. Left: maximum of the measured absolute departure of the eigenvalues of the solution matrix $X$, compared to the initial eigenvalues. Center: relative departure of the enstrophy, normalized by the initial value.
        Right: relative departure of the Hamiltonian, normalized by the initial value.}
        \label{fig:euler_pres}
    \end{figure}
    \begin{figure}
        \centering
        \includegraphics[width=0.5\textwidth]{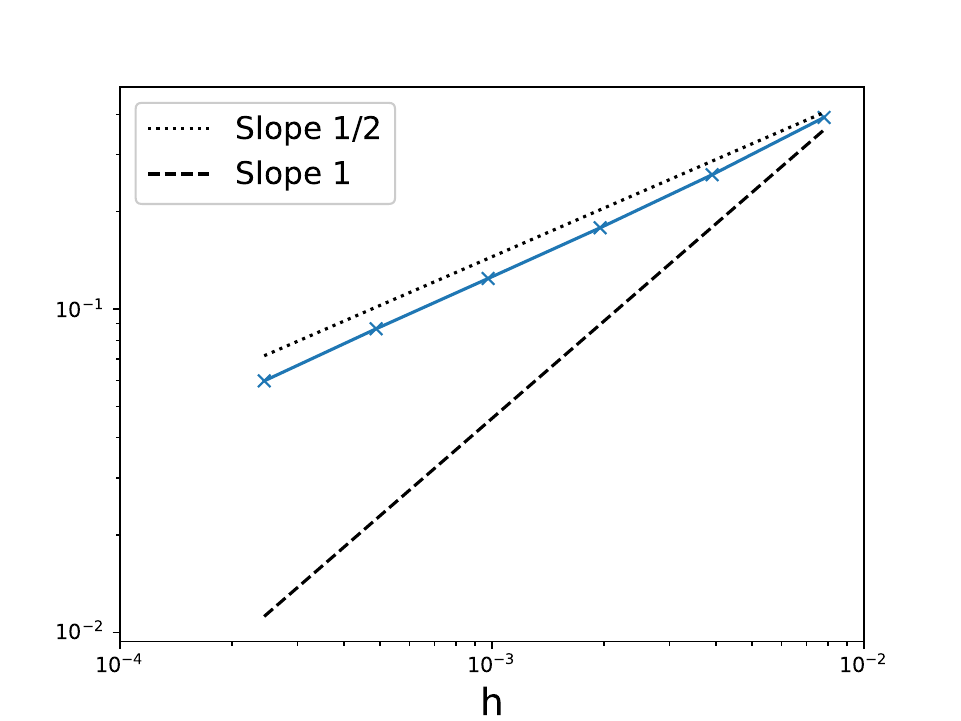}
        \caption{Strong errors for the Euler--Zeitlin system using 500 independent realizations.}
        \label{fig:euler_strong}
    \end{figure}

\section{Conclusion and outlook}
\label{sec:conclusions}
In this paper, we derived a numerical integration scheme for stochastic Lie--Poisson (LP) equations that admit an isospectral formulation. The integrator preserves the coadjoint orbits and the Casimir functions to machine precision. The LP integrator is derived through discrete Lie--Poisson reduction of the symplectic midpoint method for stochastic Hamiltonian systems. The geometric approach used in the derivation of the integrator facilitated proving the desired conservation properties as well as convergence rates. Specifically, the LP integrator converges with root mean squared order $1/2$ and weak order $1$. These properties were further illustrated in a series of numerical experiments.

The key benefit of the presented integration scheme is its scalability to high-dimensional LP systems. No computationally expensive algebra-to-group maps are used, enabling the use of the integrator for high-dimensional systems and thereby widening the range of possible applications. Examples include performing numerical studies of complex physical systems, such as stochastically forced fluid flows to study scaling limits \cite{flandoli2020convergence, flandoli2021scaling}, and spin systems \cite{mclachlan2014symplectic} to perform structure-preserving simulations with uncertainty quantification.

\section*{Acknowledgments}
We wish to express our gratitude to David Cohen, Klas Modin, Michael Roop and Milo Viviani for the valuable input and insightful discussions that helped improve the paper.
The computations were enabled by resources provided by Chalmers e-Commons at Chalmers.
This work was supported in part by the Swedish Research Council (VR) through grant no.\ 2020-04170 and grant no. 2022-03453, by the Wallenberg AI, Autonomous Systems and Software Program (WASP) funded by the Knut and Alice Wallenberg Foundation, by the Chalmers AI Research Centre (CHAIR), by the European Union (ERC, StochMan, 101088589), and by the Dutch Science Foundation (NWO) grant VI.Vidi.213.070. 
\bibliographystyle{plain}
\bibliography{refs}

\appendix
\section{Reconstruction of the stochastic Hamiltonian system}\label{app:reconstruction_stoch_hamiltonian}
In this appendix, we elaborate on the reconstruction of the stochastic Hamiltonian system~\eqref{eq:hamsystem} from the stochastic Lie--Poisson system \eqref{eq:sys_strat}. We recall the stochastic Lie--Poisson system
\begin{equation*}
\dd X_t =[\nabla H_0(X_t)^*,X_t] \, \dd t	+ \sum_{k=1}^M [\nabla H_k(X_t)^*,X_t] \circ \dd W_t^k,
\end{equation*}
where $X_t = \mu(Q_t, P_t)$. The variable $X_t\in \mathfrak{g}^*$ is a random variable, thus $Q_t$ and $P_t$ are also random variables. Since $X_t\in \mathfrak{g}^*$ for all noise realizations by construction, the reconstruction equations can be applied. The reconstruction equations to obtain an integral curve in $T^* G$ read
\begin{align*}
    \dd Q_t &= Q_t\left(\nabla H_0(X_t) \, \dd t + \sum_{k=1}^M \nabla H_k(X_t)\circ\text{d}W_t^k \right),\\
    P_t &= T_{Q_t}^*L_{Q_t^{-1}}X_t = (Q_t^{-1})^*X_t.
\end{align*}

Here $T_{Q_t}^* L_{Q_t^{-1}}X_t$ is the cotangent lift of the left action of the group. We refer to \cite{holm2009geometric, marsden2013introduction} for a full description of these equations. For completeness, we show that $T_{Q_t}^*L_{Q_t^{-1}}X_t = (Q_t^{-1})^*X_t$ in the definition of $P_t$. We let $c(s)\in G$ be a curve with $c(0)=Q_t$ and $\frac{\dd c(s)}{\dd s}|_{s=0}=v\in T_{Q_t}G$. 
Note that $Q_t \in G$ for all noise realizations, since $X_t \in \mathfrak{g}^*$ for all noise realizations. Thus, the curve $c(s)$ has a stochastic initial condition $c(0)=Q_t$ but is a deterministic function of $s$. Once $Q_t$ is known, the derivative w.r.t. $s$ at $s=0$ is therefore computed as normally.
We thus find  
\begin{align*}
\begin{split}
    \langle T_{Q_t}^* L_{Q_t^{-1}} X_t, v \rangle &= \langle X_t, T_{Q_t}L_{Q^{-1}_t}v\rangle \\
    &= \left\langle X_t, \frac{\dd}{\dd s}\left(Q^{-1}_t c(s)\right)\Big|_{s=0} \right\rangle = \langle X_t, Q^{-1}_t v\rangle = \langle (Q^{-1})^*X_t, v\rangle, 
\end{split}
\end{align*}
establishing the definition of $P_t$. The evolution of $P_t$ follows from 
\begin{align*}
    \begin{split}
        \dd P_t =& ~\dd (Q_t^{-1})^* X_t + (Q_t^{-1})^*\dd X_t = -(Q_t^{-1})^*\dd Q_t^* (Q_t^{-1})^* X_t + (Q_t^{-1})^* \dd X_t \\
        =& -(Q_t^{-1})^* \left(\nabla H_0(X_t) \dd t + \sum_{k=1}^M \nabla H_k(X_t)\circ\text{d}W_t^k \right)^* Q_t^* (Q_t^{-1})^* X_t \\
        & + (Q_t^{-1})^* [\nabla H_0(X_t)^*,X_t] \, \dd t	+ \sum_{k=1}^M [\nabla H_k(X_t)^*,X_t] \circ \dd W_t^k, \\
        =& -(Q_t^{-1})^* \left(\nabla H_0(X_t) \dd t + \sum_{k=1}^M \nabla H_k(X_t)\circ\text{d}W_t^k \right)^* X_t \\ 
        & + (Q_t^{-1})^*\left[\nabla H_0(X_t)^* \dd t + \sum_{k=1}^M \nabla H_k(X_t)^*\circ\text{d}W_t^k, X_t\right] \\
        =& -(Q_t^{-1})^* X_t \left(\nabla H_0(X_t) \dd t + \sum_{k=1}^M \nabla H_k(X_t)\circ\text{d}W_t^k \right)^* \\
        =& -P_t \left(\nabla H_0(X_t) \dd t + \sum_{k=1}^M \nabla H_k(X_t)\circ\text{d}W_t^k \right)^*,
    \end{split}
\end{align*}
thereby obtaining the stochastic Hamiltonian system \eqref{eq:hamsystem}.

\end{document}